\documentclass[12pt, a4paper]{amsart}

\usepackage{amsmath, amsthm, amssymb}
\usepackage{graphicx}
\usepackage{enumerate}
\usepackage[top=2.2cm, bottom=2.1cm, left=2.1cm, right=2.2cm]{geometry}

\newcommand{\A}{\mathbf A}
\newcommand{\B}{\mathbf B}

\newcommand{\x}{\mathbf x}
\newcommand{\y}{\mathbf y}

\newtheorem{definition}{Definition}
\newtheorem{lemma}{\bf Lemma}
\newtheorem{proposition}{\bf Proposition}
\newtheorem{theorem}{\bf Theorem}
\newtheorem{corollary}{\bf Corollary}

\renewenvironment{proof}{\noindent {\bf Proof: }}{\rm\\}
\theoremstyle{definition}
\newtheorem{remark}{Remark}{\rm}
\newtheorem{example}{Example}{\rm}

\usepackage{color}

\usepackage[T1]{fontenc}

\usepackage{algorithm, algpseudocode}

\makeatletter
\renewcommand{\p@algorithm}{\arabic{algorithm}\expandafter\@gobble}
\makeatother

\newcounter{step}[algorithm]
\newcommand\STEP[2][\(\triangleright\)]{%
	\refstepcounter{step}
	\vskip 0.25\baselineskip
	\item[]\hskip -\algorithmicindent #1 \textbf{Step \arabic{step}}%
	\ifthenelse{\equal{\unexpanded{#2}}{}}{}{ (\texttt{#2})}%
	\textbf{.}%
}
\newenvironment{algo}{\algo}{}
\def\algo#1\end{%
	\noindent\fbox{%
	\begin{minipage}[b]{\dimexpr\columnwidth-\algorithmicindent\relax}
	\begin{algorithmic}
	#1
	\end{algorithmic}
	\end{minipage}
	}%
\end}

\makeatletter
\renewcommand{\fnum@algorithm}{\fname@algorithm}
\makeatother

\usepackage{mathrsfs}

\begin{document}
\pagestyle{plain}

\title{Alternating projections with applications to Gerchberg-Saxton error reduction$^1$}
\author{Dominikus Noll$^{2}$}
\thanks{$^1$Dedicated to R.T. Rockafellar on the occasion of his 85th anniversary}
\thanks{$^2$Institut de Math\'ematiques, Universit\'e de Toulouse, France}
\date{}

\begin{abstract}
We consider convergence of alternating projections between non-convex sets and obtain applications to convergence of
the Gerchberg-Saxton error reduction method, of the Gaussian expectation-maximization algorithm, and of Cadzow's algorithm.
\\

\noindent
{\bf Key words:} 
Alternating projections $\cdot$ subanalytic sets $\cdot$ phase retrieval $\cdot$ Gerchberg-Saxton $\cdot$ Douglas-Rachford
$\cdot$ Gaussian EM-algorithm $\cdot$ Cadzow algorithm
\end{abstract}
\maketitle

\section{Introduction}
We consider convergence of
alternating projections $a_k \in P_A(b_{k-1})$, $b_k\in P_B(a_k)$ between closed sets $A,B\subset \mathbb R^n$, where
$P_A,P_B$ are the potentially set-valued orthogonal projectors on $A,B$. 
Since their invention
\cite{schwarz} alternating projections have been understood as an algorithmic solution to the
feasibility problem of finding points $x^*\in A \cap B$.  In the infeasible
case $A \cap B=\emptyset$, alternating projections are still interpreted  as of providing generalized solutions
realizing the gap between $A$ and $B$. 

It is well-known \cite{bauschke-survey} that bounded alternating sequences converge
if $A,B$ are closed convex, while convergence may fail already if one of the sets is
non-convex. If $a_k,b_k$ are bounded and satisfy $a_k-a_{k-1}\to 0$,
$b_k-b_{k-1}\to 0$  as $k \to \infty$, then by Ostrowski's theorem  the sets $A^*,B^*$ of accumulation points of the $a_k,b_k$ are compact continua. 
This includes the 
singleton case $A^*=\{a^*\}, B^*=\{b^*\}$ with convergence, but allows examples where
$A^*,B^*$ are non-singleton. The first cases of failure of convergence with non-singleton $A^*=B^*\subset A \cap B$
were constructed in \cite{bauschke-noll} and 
\cite{douglas}.

In the feasible case $A \cap B \not=\emptyset$ local convergence of alternating projections was
established under transversality hypotheses
in \cite{malick,luke,bauschke1,bauschke2,obsolete,hesse}, where the speed of convergence is linear.
Convergence for cactus sets without transversality 
was proved in \cite{bauschke-noll}, and the case of tangential intersection was addressed in \cite{initial,noll}.
General convergence conditions are given in \cite{csiszar}, but are difficult to check in practice. 
The Kurdyka-\L ojasiewicz (KL) circle of ideas plays a crucial role in the approach \cite{noll},
and there had previously been results for related projection based methods in
\cite{attouch}. In \cite{chinese} the approach of \cite{noll} and the KL-property is used to address the infeasible case,
where the authors do not focus on geometric properties of the sets $A,B$, but
on properties of the sequence $a_k,b_k$ directly.

In this work we show that the infeasible case can be covered by suitably adapting the approach of \cite{noll}.
This gives convergence under geometric conditions
in terms of $A,B$.

A central concern of this work is application of alternating projections to
the Gerchberg-Saxton error reduction method \cite{gerchberg},  introduced  in 1972. This classical tool for phase retrieval has been used 
successfully 
for more than 40 years without convergence certificate.
The first
convergence proof ever appeared in 2013 in \cite{initial,noll}, 
addressing the feasible case and including subanalytic sets.
Here we give the  first convergence proof covering also the infeasible case, providing criteria which can often be
checked in practice.

It turns out that not only had Gerchberg-Saxton error reduction been used without theoretical convergence certificates for decades,
neither had the question ever been raised whether there could be cases where convergence fails.
We therefore supplement a first counterexample, showing that Gerchberg-Saxton
error reduction may indeed fail to converge even in the feasible case if only the prior information set is sufficiently irregular. 

We end with a glimpse on the EM-algorithm, where the situation is not unlike in phase retrieval, inasmuch
as since the 1970s a satisfactory convergence theory outside the realm of convexity is missing. For variants of the
EM-algorithm which are realizations of alternating projections, we can prove convergence without convexity.
Our findings also concern the speed of convergence, which is shown to be sublinear.

The structure of the paper is as follows. After the preparatory Sections \ref{sect_prep}, \ref{sect_local},  Sections \ref{sect_angle}, \ref{sect_holder}, \ref{sect_reach}, \ref{sect_three}
adapt notions developed for the feasible case in \cite{noll} to address the infeasible case. Section \ref{sect_convergence}
gives the central convergence result. Gerchberg-Saxton error reduction is discussed in Section \ref{sect_GS},
counterexamples for the Gerchberg-Saxton and Hybrid-Input-Output (HIO) algorithms are constructed in Sections \ref{spiral}, \ref{sect_hio}. The Gaussian
EM-algorithm is given attention in Section \ref{EM}, and Cadzow's algorithm in Section \ref{sect-cadzow}.

\section*{Notation}
Notions from nonsmooth analysis are covered by \cite{rock,mord}. Euclidean balls are denoted
$\mathcal B(x,\delta)$, and $\mathcal N(A,\delta)=\{x\in \mathbb R^n: d_A(x)\leq \delta\}$ is the Euclidean
$\delta$-neighborhood of a set $A$.
The proximal normal
cone to $A$ at $a\in A$ is 
$N_A^p(a)=\{\lambda u: \lambda \geq 0, a \in P_A(a+u)\}$,
the normal cone 
is the set 
$N_A(a)$ of $v$ for which there exist $a_k\in A$ with $a_k \to a$ and $v_k\in N_A^p(a_k)$
such that $v_k\to v$. The Fr\'echet normal cone $\widehat{N}_A(a)$ to $A$ at $a\in A$
is the set of $v$ for which $\limsup_{A \ni a'\to a} \frac{\langle v,a'-a\rangle}{\|a'-a\|}\leq 0$; cf. \cite[(1.2)]{mord}.
We have $N_A^p(a)\subset \widehat{N}_A(a)\subset N_A(a)$; cf. {\cite[Chapter 2.D and (1.6)]{mord} or \cite[Lemma 2.4]{bauschke1}}.
The proximal subdifferential $\partial_pf(x)$ of a lower semi-continuous
function $f$ at $x\in {\rm dom}f$ is the set of vectors 
$v\in \mathbb R^n$ such that $(v,-1)\in N^p_{{\rm epi}f}(x,f(x))$; \cite[(2.81)]{mord}. The 
subdifferential $\partial f(x)$ of $f$ at $x\in {\rm dom} f$ is the set of $v$
satisfying $(v,-1)\in N_{{\rm epi} f}(x,f(x))$. 
The Fr\'echet subdifferential 
$\widehat{\partial}f(x)$ at $x\in {\rm dom} f$ is the set of $v\in \mathbb R^n$ such that
$(v,-1)\in \widehat{N}_{{\rm epi}f}(x,f(x))$, cf.  \cite[(1.51)]{mord}. The indicator function  of a set $A$ is $i_A$,
the distance to $B$ is $d_B$. We have the following

\begin{lemma}
\label{first}
Let $r^*\geq 0$,  $f = i_A + \frac{1}{2}(d_B - r^*)^2$, $a^+\in A$, $v = \lambda(b-a^+) \in N_A^p(a^+)$, where $b \in B$, $\lambda \geq 0$. Then
$v + \frac{d_B(a^+)-r^*}{d_B(a^+)}(a^+-P_B(a^+) )\subset \widehat{\partial} f(a^+)$.
\hfill $\square$
\end{lemma}

\begin{proof}
By \cite[Cor. 1.96]{mord} or \cite[p. 340]{rock} we have $\frac{a^+-P_B(a^+)}{\|a^+-P_B(a^+)\|} \in \widehat{\partial} d_B(a^+)$, hence by the chain rule
$(d_B(a^+)-r^*)  \frac{a^+-P_B(a^+)}{\|a^+-P_B(a^+)\|} \in \widehat{\partial} \frac{1}{2}(d_B-r^*)^2 (a^+)$. Since
$\widehat{\partial}{i}_A(a^+)=\widehat{N}_A(a^+)$ by \cite[Prop. 1.79]{mord}, we have $v\in N^p_A(a^+)\subset \widehat{N}_A(a^+)\subset \widehat{\partial}i_A(a^+)$,
and by the sum rule \cite[Lemma 2.4]{kruger} we have $\widehat{\partial} i_A(a^+) + \widehat{\partial} \frac{1}{2}(d_B-r^*)^2(a^+) \subset \widehat{\partial} f(a^+)$,
which completes the proof.
\hfill $\square$
\end{proof}

The importance of $f$ in KL-theory is well-known. See for instance \cite{attouch,bolte_new,csiszar,noll,obsolete}.

\section{Preparation}
\label{sect_prep}
Given nonempty closed sets $A,B\subset \mathbb R^n$, we consider sequences of alternating projections
$b_k\in P_B(a_k)$, $a_{k+1}\in P_A(b_k)$, where $P_A,P_B$ are the possibly set-valued orthogonal projectors on $A,B$.
We use the notation
\[
a_k \to b_k \to a_{k+1}, \quad b_{k-1} \to a_k \to b_k
\]
for the building blocks of the alternating sequence, and sometimes the index free notation
$a \to b \to a^+$ and $b \to a^+ \to b^+$ introduced in \cite{noll}. If a projection is single-valued, we write $b=P_B(a)$.

For a bounded alternating sequence $a_k \to b_k \to a_{k+1}$ let $A^*$, $B^*$ be the set of accumulation points of the $a_k$, $b_k$,
and  $r^* = \inf\{ \|a_k-b_k\| : k\in \mathbb N\}$,  then we call
$(A^*,B^*,r^*)$ the {\em gap of the alternating sequence}. For every $a^*\in A^*$ there exists $b^*\in B^* \cap P_B(a^*)$ with $\|a^*-b^*\|=r^*$, and
vice versa, for every $b^*\in B^*$ we find $a^*\in A^* \cap P_A(b^*)$ with $\|b^*-a^*\| = r^*$. 
We are interested in those cases where the sequences $a_k,b_k$ converge $a_k \to a^*$, $b_k\to b^*$ i.e.,
$A^* = \{a^*\}$, $B^* = \{b^*\}$.  In the alternative, if this fails,  we would hope that at least one of the sequences
converges. The case $r^*=0$ treated in \cite{noll}
is referred to as the  feasible  case. Here convergence of one of the sequences $a_k$ or $b_k$
implies convergence of the other, but this may no longer be true in the infeasible case $r^*>0$.

In \cite{malick}, and subsequently  in
\cite{luke,bauschke1,bauschke2,obsolete,hesse,noll}, the following point of view is taken: Given a point $x^* \in A \cap B$, find conditions
under which {\em any} alternating sequence, once it gets sufficiently close to $x^*$, is captured and forced to converge to some point in the intersection.
Here we investigate under which conditions a similar local attraction phenomenon may occur in the infeasible case $r^*>0$. 

Given subsets $A^*\subset A$, $B^*\subset B$ and $r^*\geq 0$, we say that
$(A^*,B^*,r^*)$ is a gap between $A$ and $B$, or simply a gap, 
if for every $a^*\in A^*$ there exists $b^*\in B^*$ with $b^*\in P_B(a^*)$ and $\|a^*-b^*\|=r^*$, and
vice versa, for every $b^*\in B^*$ there exists $a^*\in A^*$ with $a^*\in P_A(b^*)$ and $\|a^*-b^*\|=r^*$. 
The question is then the following: Suppose an alternating sequence gets close to that gap
in the sense that $a_k$ is close to $A^*$, $b_k$ is close to $B^*$, and $r^* < \|a_k - b_k\| < r^*+\eta$ for some small $\eta > 0$,
will this sequence be captured and forced to
converge $a_k \to a^*$, $b_k \to b^*$, with $\|a^*-b^*\| = r^*$, realizing that gap?

\section{Local alternating projections}
\label{sect_local}
Despite the absence of a satisfactory convergence theory, non-convex alternating projections had been used on a 
purely experiment basis for decades. With \cite{noll} many of these heuristics have now a sound theoretical basis,
but occasional experiments would suggest to go a little further and include cases, where
projections are computed only locally. This point of view will now be given consideration.

We say that $a^+ \in A$ is a local projection of $b\in B$ onto $A$ if there exists a neighborhood $V$ of $a^+$ such that
$a^+ \in P_{A\cap V}(b)$. In other words, there might be points in $A$ closer to $b$ than $a^+$, but not in the neighborhood $V$ of $a^+$.
Now in this situation there exists a point $c \in (b,a^+)$, sufficiently close to $a^+$, 
such that $a^+ = P_A(c)$. But then by Lemma \ref{first}, 
$v + \frac{d_B(a^+)-r^*}{d_B(a^+)}(a^+-P_B(a^+) )\subset \widehat{\partial} f(a^+)$,
where as before
$f = i_A+\frac{1}{2}(d_B-r^*)^2$. Since $a^++\mathbb R^+ (b-a^+) = a^++\mathbb R^+(c-a^+)$, we have
$\lambda(b-a^+) + \frac{d_B(a^+)-r^*}{d_B(a^+)}(a^+-P_B(a^+) )\subset \widehat{\partial} f(a^+)$
for every $\lambda \geq 0$. 
In consequence, we have the following extension of Lemma \ref{first}:

\begin{lemma}
Suppose $a^+\in A$ is a local projection from $b\in B$, and $b^+\in P_B(a^+)$. Then
$\lambda(b-a^+)+\frac{d_B(a^+)-r^*}{d_B(a^+)}(a^+-b^+) \in \widehat{\partial} f(a^+)$. 
\hfill $\square$
\end{lemma}

\begin{definition}
A sequence $a_k\in A$, $b_k\in B$ with $\|a_k - b_{k-1}\| \leq \|a_{k-1}-b_{k-1}\|$, $b_k\in P_B(a_k)$, 
and  $a_k$ 
a local projection of $b_{k-1}$, noted
\begin{equation}
\label{local}
a_{k} \to b_{k} \stackrel{\ell}{\to} a_{k+1}, \quad b_{k-1} \stackrel{\ell}{\to} a_k \to b_k,
\end{equation}
is called a local
alternating sequence of projections.
\end{definition}

\begin{remark}
Our definition of local alternating sequence $a_k\to b_k\stackrel{\ell}{\to} a_{k+1}$ has to require that the distance is decreasing, while this is automatically
true for traditional alternating sequences. Note also that (\ref{local}) breaks the symmetry between $A$ and $B$.
\end{remark}

\begin{remark}
The definition of a local projection is convenient, because in applications the projection
on one of the sets often requires solving a non-linear and non-convex optimization
program $\min \{d_A(b): b \in B\}$, and finding a global minimum might be hard. On the other hand,
a local solver using a descent method started at the last projected point $b\in B$ will obviously
lead to a local projection $a^+\in A$ satisfying $\|a^+-b\| < \|a-b\|$. 
Naturally, for convex $A$ local projections are just ordinary projections.
\end{remark}

Lemma \ref{first} suggest going even one step further. 
We do not need $a^+\in A$ to be a local projection from $b\in B$. What is needed is $b-a^+\in N_A^p(a^+)$.
This leads to the following:

\begin{definition}
A sequence $a_k\in A$, $b_k\in B$ with
$\|b_{k-1}-a_k\| \leq \|b_{k-1}-a_{k-1}\|$, $b_k\in P_B(a_k)$, and $b_{k-1}-a_k \in N_A^p(a_k)$
is called a prox-alternating sequence of projections, noted
\begin{equation}
    \label{prox-building}
a_{k}\to b_{k} \stackrel{p}{\to} a_{k+1}, \quad b_{k-1} \stackrel{p}{\to} a_k \to b_k.
\end{equation}
\end{definition}

\begin{remark}
Clearly every alternating sequence is local alternating, and every local alternating sequence is
a prox-alternating. For convex $A,B$ those all coincide.
\end{remark}

\begin{remark}
Let $a_k\to b_k \stackrel{p}{\to} b_{k+1}$ be a bounded prox-alternating sequence, 
$A^*,B^*$ the sets of accumulation points of the $a_k,b_k$
with gap value $r^*=\inf \{ \|a_k-b_k\|: k \in \mathbb N\}$.  Define
\begin{equation}
    \label{s}
A^s = \{a_k: k \in \mathbb N\} \cup A^*, \quad B^s = \{b_k: k\in \mathbb N\} \cup B^*.
\end{equation}
Then $a_k,b_k$ is converted into a traditional alternating sequence between $A^s,B^s$, where 
$P_{B^s}(a_k) = b_k\in P_B(a_k)$, but where the projection $P_{A^s}(b_{k-1}) = P_{A\cap V}(b_{k-1}) = a_k$, which was local for $A$,
is now rendered global for $A^s$, because points in
$A$ which might make the projection $b_{k-1} \to a_k$ a local one have been removed from $A^s$. We continue to
call $(A^*,B^*,r^*)$ the gap of the prox-alternating sequence.
\end{remark}

\begin{theorem}
\label{critical}
Let $a_k\in A,b_k\in B$ be a bounded prox-alternating sequence with gap $(A^*,B^*,r^*)$. 
Then every $a^*\in A$ is a critical point of $f=i_A+\frac{1}{2}(d_B-r^*)^2$.
When $r^* > 0$ and  $b_{k-1}-b_k \to 0$, then 
$a^*\in A^*$ is also a critical point of $g=i_A+\frac{1}{2}d_B^2$. 
\end{theorem}

\begin{proof}
Every $a^*$ is a global minimum of $f$, hence a critical point.
Consider $g$ for the case $r^*>0$.
From Lemma \ref{first} we get
$b_{k-1}-a_k + a_k- P_B(a_k) \subset \widehat{\partial} g(a_k)$. 
Select an infinite subsequence $k\in \mathcal K$ such that $b_{k-1}\to b^*$, $a_k\to a^*$, $k\in \mathcal K$,
then also $b_k \to b^*$, using the hypothesis $b_{k-1}-b_k\to 0$. 
Then $b_{k-1} - a_k + a_k - b_k \in b_{k-1}-a_k + a_k - P_B(a_k) \subset \widehat{\partial} g(a_k)$,
hence $0=b^*-a^*+a^*-b^* \in \partial g(a^*)$, where $\partial g(a^*)$ is the limiting subdifferential. 
\hfill $\square$
\end{proof}

\section{Angle condition}
\label{sect_angle}
We extend the angle condition introduced in \cite{noll} for the feasible case to the general 
case $r^* \geq 0$ and to prox-alternating sequences.

\begin{definition}
{\bf (Angle condition)}.
We say that the gap $(A^*,B^*,r^*)$ satisfies the angle condition with constant $\gamma > 0$ and exponent
$\omega \in [0,2)$, if there exist neighborhoods
$U$ of $B^*$ and $V$ of $A^*$ such that for every building block $b\stackrel{p}{\to} a^+\to b^+$ with $r=\|a^+-b^+\| > r^*$
and $a^+\in V$, $b^+\in U$, the estimate
\begin{equation}
\label{angle1}
\frac{1-\cos \alpha}{(r-r^*)^\omega} \geq \gamma
\end{equation}
holds
for the angle $\alpha = \angle(b-a^+,b^+-a^+)$.
\end{definition}

\begin{remark}
The interpretation of (\ref{angle1}) is that if the angle $\alpha$ between consecutive projection steps wants to get
close to 0 as the alternating sequence approaches the gap, then this
decrease has to be controlled by the speed with which the alternating sequence approaches the gap value $r^*$. Condition (\ref{angle1})
is strongest for $\omega=0$, and becomes  less binding as $\omega$ approaches 2. Values beyond 2 are too weak to be of interest.
The case $\omega=0$ is allowed, and here the angle
$\alpha$ stays away from $0$. 
\end{remark}

\begin{remark}
In \cite{noll} the condition was formulated for the feasible case $(\{x^*\},\{x^*\},0)$, where $x^*\in  A \cap B$.
Note that the angle condition breaks the symmetry. If we want to use the corresponding
condition for building blocks $a\stackrel{p}{\to} b\to a^+$, then we have to refer to a gap $(B^*,A^*,r^*)$.
\end{remark}

\begin{remark}
In the feasible case the sets $A,B$ intersect at $x^*$, and in \cite{noll} the term separable intersection, or intersection {\it at an angle},  was employed synonymously with
the term angle condition with $\omega=0$.
One could also refer to this as tangential intersection, as opposed to transversal intersection, or intersection
{\em at an angle}. 
\end{remark}

\begin{definition}
\label{def_loja}
{\bf (\L ojasiewicz inequality)}.
Let $f:\mathbb R^n \to \mathbb R \cup\{\infty\}$ be lower semi-continuous with closed domain
such that $f|_{{\rm dom} f}$ is continuous. We say that $f$ satisfies the \L ojasiewicz inequality
with exponent $\theta \in [0,1)$ at the critical point
$x^*$ of $f$ if there exists $\gamma > 0$, $\eta > 0$, and a neighborhood $V$ of $x^*$ such that
$(f(x)-f(x^*))^{-\theta} \|g\| \geq \gamma$ for every $x\in V$ with $f(x^*) < f(x) < f(x^*)+\eta$ and every $g\in {\partial} f(x)$.
\end{definition}

Here $x^*$ is critical in the sense of the limiting
subdifferential, see \cite{mord,rock}.   Note that we expect values $\theta \in [\frac{1}{2},1)$. Indeed, consider a real-analytic function $f$ of one variable
with a critical point at $x^*$. If $f'(x^*)=\dots= f^{(N)}(x^*)=0$, $f^{(N+1)}(x^*) \not=0$, then the \L ojasiewicz inequality holds with $\theta=N/(N+1)$, so the best possible value is
$\theta = \frac{1}{2}$ for $N=1$.

\begin{remark}
Suppose $K^*$ is a compact set of critical points of $f$ with
$f(K^*)$ constant on $K^*$. If $f$ satisfies the \L ojasiewicz inequality at every $x^*\in K^*$, then
by a simple compactness argument there exists a neighborhood $V$ of $K^*$ and parameters $\theta$ and $\gamma,\eta > 0$, valid for the whole of $K^*$, 
for which the same estimate is satisfied.
\end{remark}

Let $(A^*,B^*,r^*)$ be a gap and
$a^*\in A$, $b^*\in B$ with $b^*-a^*\in N_A^p(a^*)$,  $b^*\in P_B(a^*)$. Let
$f=i_A+\frac{1}{2}(d_B-r^*)^2$, then  by  Lemma \ref{first} 
$a^*$ is a critical point of $f$. Since the domain $A$ of $f$ is closed,
$f$ is amenable to Definition \ref{def_loja}. We deduce the following


\begin{lemma}
\label{second}
Let $(A^*,B^*,r^*)$ be a gap with compact $A^*$ and
suppose $f = i_A+\frac{1}{2}\left(d_B-r^*\right)^2$ satisfies the \L ojasiewicz
inequality with exponent $\theta \in [0,1)$ and constant $\gamma >0$ on $A^*$. 
Then $\theta \geq \frac{1}{2}$, and there exists a neighborhood $V$ of $A^*$ and $\eta > 0$
such that for every prox-building block $b\stackrel{p}{\to} a^+\to b^+$ with $a^+\in V$ and $r^* < r = \|a^+-b^+\| < r^* +\eta$ the angle condition 
\begin{equation}
    \label{angle}
\frac{1-\cos\alpha}{\left( \|a^+-b^+\|-r^{*}\right)^{4\theta-2}} \geq \gamma 
\end{equation}
is satisfied, where
$\alpha = \angle(b-a^+,b^+-a^+)$. 
\end{lemma}

\begin{proof}
The function $f = i_A + \frac{1}{2}(d_B-r^*)^2$ has constant value $0$ on $A^*$. By the definition of the \L ojasiewicz inequality there exists a neighborhood
$V$ of $A^*$ and $\gamma > 0$ such that every $a^+\in A \cap V$ with $r^* < d_B(a^*) < r^*+\eta$ satisfies
$$
f(a^+)^{-\theta} {\rm dist}\left(0,\widehat{\partial} f(a^+)\right) \geq \gamma.
$$
By Lemma \ref{first} this means
$$
2^\theta \left( d_B(a^+)-r^* \right)^{-2\theta} \left\|\lambda(b-a^+) + (d_B(a^+)-r^*) \frac{a^+-P_B(a^+)}{\|a^+-P_B(a^+)\|}\right\| \geq \gamma
$$
for every $\lambda \geq 0$. We deduce using  the substitution $\mu = \lambda \frac{\|a^+-P_B(a^+)\|}{d_B(a^+)-r^*}$ that
for every $b^+\in P_B(a^+)$
\begin{equation}
\label{new_ang}
2^\theta \frac{\left( d_B(a^+)-r^* \right)^{-2\theta+1}}{\|a^+-P_B(a^+)\|} \min_{\mu \geq 0} \left\|\mu(b-a^+) +  a^+-b^+\right\| \geq \gamma.
\end{equation}
Assume that the angle
$\alpha = \angle(b-a^+,b^+-a^+)$ is smaller than $90^\circ$, then the minimum in (\ref{new_ang}) is $\|a^+-b^+\|\sin \alpha$. Hence
$$
\frac{\sin \alpha}{\left(d_B(a^+)-r^*\right)^{2\theta-1}} \geq 2^{-\theta} \gamma.
$$
Since $1-\cos \alpha \geq \frac{1}{2} \sin^2\alpha$, we obtain
\begin{equation}
\label{29}
\frac{1-\cos \alpha}{(d_B(a^+)-r^*)^{4\theta-2}} \geq 2^{-2\theta-1} \gamma^2.
\end{equation}

Now for angles $\alpha > 90^\circ$ we have $\cos \alpha < 0$, hence $1-\cos \alpha > 1$. The minimum in (\ref{new_ang}) is now attained at
$\mu=0$,  with value $\|a^+-b^+\|$. Hence (\ref{new_ang}) implies $(d_B(a^+)-r^*)^{1-2\theta} \geq 2^{-\theta}\gamma$, 
hence $(d_B(a^+)-r^*)^{2-4\theta} \geq 2^{-2\theta} \gamma^2 > 2^{-2\theta-1} \gamma^2$, so that 
(\ref{29}) holds also in this case.
\hfill $\square$
\end{proof}

\begin{remark}
We do not 
expect exponents better than $\theta =\frac{1}{2}$  in Definition \ref{def_loja}, and hence in (\ref{angle}), and due to
$\omega = 4\theta-2$  this corresponds
to the best value $\omega=0$ in (\ref{angle1}).
As we shall later see, in the case $r^* > 0$ we even expect values $\theta \in [\frac{3}{4},1)$, or in terms of
(\ref{angle1}), values $\omega \geq 1$.
\end{remark}

\begin{remark}
For the best possible $\theta=\frac{1}{2}$ the denominator in (\ref{angle}) equals 1, so that the condition requires 
$\alpha$ to stay away from $0$. Here we expect linear convergence, and that will be proved in Theorem \ref{rate}. 
In the feasible case $r^*=0$ this was referred to in \cite{noll} as separable intersection.
\end{remark}

We now apply our findings to subanalytic sets. Recall that  $A \subset \mathbb R^n$ is
{\it semi-analytic} if for every $x'\in \mathbb R^n$
there exists an open neighborhood $V$ of $x'$ such that
\begin{equation}
\label{subanalytic}
A\cap V = \bigcup_{i\in I} \bigcap_{j\in J} \{x\in V: \phi_{ij}(x)=0,\psi_{ij} > 0\}
\end{equation}
for finite sets $I,J$ and real analytic functions $\phi_{ij},\psi_{ij} :V \to \mathbb R$.
A set $B\subset \mathbb R^n$ is {\it subanalytic} if for every $x'\in \mathbb R^n$ there exists a 
neighborhood $V$ of $x'$ and a bounded
semi-analytic set $A\subset \mathbb R^n \times \mathbb R^m$ for some $m$ such that
$B\cap V = \{x\in \mathbb R^n: \exists y\in \mathbb R^m \, (x,y)\in A\}$.
A function $f:\mathbb R^n \to \mathbb R \cup \{\infty\}$ is subanalytic if its graph is a subanalytic set in $\mathbb R^n \times \mathbb R$.

\begin{corollary}
Let $A,B$ be subanalytic sets, and let $a_k,b_k$ be a bounded prox-alternating sequence with gap $r^*$. Then
there exists an exponent $\theta\in [\frac{1}{2},1)$ and a constant
$\gamma > 0$, such that 
\[
\frac{1-\cos\alpha_k}{\left(\|a_k-b_k\|-r^{*}\right)^{4\theta-2}}\geq \gamma
\]
for $k\in \mathbb N$.
\end{corollary}

\begin{proof}
Let $(A^*,B^*,r^*)$ be the gap of the alternating sequence. Then by Theorem \ref{critical} every $a^*\in A^*$ is a critical point
of $f = i_A+\frac{1}{2} (d_B-r^*)^2$. Since $A,B$ are subanalytic, so is $f$ (cf. \cite[Thm. 3]{noll}), and by \cite[Thm. 3.1]{bolte1} $f$ satisfies the \L ojasiewicz
inequality with the same exponent $\theta \in [\frac{1}{2},1)$ throughout $A^*$. Now the result follows from Lemma \ref{second}.
\hfill $\square$
\end{proof}

\begin{remark}
Let $A,B$ be subanalytic, and consider a prox-alternating sequence $a_k,b_k$. Then trivially the angle condition (\ref{angle}) still holds 
for the gap of the sequence, but now
with regard to
the sets $A^s,B^s$ in (\ref{s}). This is significant in so far as $A^s,B^s$ are defined recursively and have no reason to
be subanalytic. 
\end{remark}

\section{H\"older regularity}
\label{sect_holder}
We extend the notion of H\"older-regularity introduced for the feasible case in \cite{noll}.

\begin{definition}
\label{holder1}
{\bf (H\"older regularity)}.
We say that a gap $(A^*,B^*,r^*)$ is $\sigma$-H\"older regular with constant $c>0$  and exponent $\sigma \in (0,1)$
if there exist a neighborhood $V$ of $A^*$ and $\eta > 0$ such that
every building block $b\to a^+\to b^+$ with  $r=\|a^+-b^+\|$, $r^* < r < r^*+\eta$ and $a^+\in V$
satisfies:
\begin{equation}
    \label{holder1}
\mathcal{B}(a^+,(1+c)r) \cap \{b\in P_A^{-1}(a^+):
\langle a^+-b^+,b-b^+\rangle > \sqrt{c}r(r-r^*)^{\sigma}
\|b-b^+\|\} \cap B = \emptyset,
\end{equation}
or what is the same with the angle $\beta = \angle(a^+-b^+,b-b^+)$:
\begin{equation}
    \label{holder}
\mathcal{B}(a^+,(1+c)r) \cap\{b\in P_A^{-1}(a^+): \cos\beta >\sqrt{c}
(r-r^*)^\sigma\}\cap B=\emptyset.
\end{equation}
\end{definition}

\begin{remark}
Note the asymmetry in the definition. If we want $A^*,B^*$ to change roles, 
we say that the gap $(B^*,A^*,r^*)$ is $\sigma$-H\"older regular.
\end{remark}

\begin{remark}
The definition agrees with the notion of $\sigma$-H\"older-regularity of $B$ with respect to $A$ at $x^*\in A \cap B$ in
\cite{noll} when we take as gap $(\{x^*\},\{x^*\},0)$. Even in the feasible case this is already an asymmetric condition.
\end{remark}

\begin{remark}
Note that in the case $r^*=0$ the notion $\sigma$-H\"older regularity with $\sigma=0$ includes a very weak form of transversality generalizing
the transversality notions in
 \cite{malick,luke,bauschke1,bauschke2,obsolete,hesse}. Consequently, linear convergence results based on our concepts of
 $0$-H\"older regularity in tandem with $0$-separability are the strongest in this class.
\end{remark}

\begin{remark}
Let $a_k,b_k$ be an alternating sequence and let $A^*,B^*$ be the corresponding sets of accumulation points.
Suppose the gap $(A^*,B^*,r^*)$ is $\sigma$-H\"older regular with constant $c > 0$. Then trivially $(A^*,B^*,r^*)$
is also $\sigma$-H\"older regular with regard to the underlying sets $A^s$, $B^s$. This simply means that (\ref{holder1}) is only
required for the elements of the alternating sequence. 
\end{remark}

\begin{definition}
\label{holder2}
{\bf (H\"older regular sequence)}.
An alternating sequence is $\sigma$-H\"older regular with constant $c>0$ if the gap $(A^*,B^*,r^*)$ of its accumulation
points is $\sigma$-H\"older regular with constant $c>0$ in the sense of Definition {\rm \ref{holder1}} with the underlying sets $A^s,B^s$.
\end{definition}

\section{Slowly shrinking reach}
\label{sect_reach}
In this section we provide a sufficient condition for H\"older regularity.
Let $b\in B$ and $d$ be an outer normal of $B$ at $b$, $d\in N_B(b)$, $d\not=0$. We define
\[
R(b,d)=\sup\{R \geq 0: P_B(b+Rd/\|d\|)=b\}
\]
and call this the reach of $B$ at $b$ along $d$. Note that $R(b,d)\in [0,\infty]$, and $R(b,d) > 0$ for a proximal normal,
i.e., if $d\in N_B^p(b)$. 
We say that $\mathcal{B}(b+R(b,d)d/\|d\|,R(b,d))$ is the largest ball with centre on the ray $b+\mathbb R_+d$ which touches the set $B$ from
outside. The case $R(b,d)=+\infty$ occurs e.g. when $B$ is convex at $b$, in which case the largest ball is the half space
$\langle x-b,d\rangle \geq 0$. If $d$ is not a proximal normal, then $R(b,d)=0$, so the largest ball is a dot.

\begin{definition}
{\bf (Slowly shrinking reach).}
Let $\sigma\in (0,1]$.
The set $B$ has $\sigma$-slowly shrinking reach with respect to $A$ and gap $(A^*,B^*,r^*)$
if there exists $0 \leq \tau < 1$ such that
\begin{eqnarray}
\label{slow}
\limsup_{r^* < \|a^+-b^+\|\to r^*} \frac{\left(\|a^+-b^+\|-r^*\right)^\sigma}{R(b^+,d)-r^*} \leq \tau,
\end{eqnarray}
where 
$d=(a^+-b^+)/\|a^+-b^+\|$, and the limit is over building blocks $b\to a^+\to b^+$ approaching the gap.
We say that the reach shrinks with exponent $\sigma$ and
rate $\tau$.
\end{definition}

In \cite{noll} this was introduced for the case $r^*=0$, where it was termed slowly vanishing reach.
The following - not surprisingly - extends \cite[Prop. 5]{noll}.

\begin{proposition}
\label{prox}
Let $\sigma\in (0,1)$, $\tau \in [0,1)$.
Suppose $B$ has $\sigma$-slowly shrinking reach with rate $\tau \geq 0$ with respect to $A$ and gap $(A^*,B^*,r^*)$.
Then that gap is $(1-\sigma)$-H\"older regular with any constant $c>0$
satisfying $\frac{\tau}{2}\sqrt{2+c} < 1$.
\end{proposition}

\begin{proof}
We choose $\tau' > \tau$ and $\epsilon > 0$ such that
$\frac{\tau'}{2}\left(\epsilon + \sqrt{\epsilon^2+2+c}\right) < 1$. By hypothesis there exists neighborhoods $U$ of $B^*$ and $V$ of $A^*$
such that $\frac{(r-r^*)^\sigma}{R(b^+,d)-r^*} < \tau'$ for every building block $b \to a^+\to b^+$ with
$b^+\in U$, $a^+\in V$ and $d=(a^+-b^+)/\|a^+-b^+\|$, $r=\|a^+-b^+\| > r^*$. By shrinking $U,V$ further if
necessary, we may arrange that $(r-r^*)^{1-\sigma}<\epsilon$. We show that the neighborhoods are
as required in (\ref{holder}).

We have to show that $b\in B$ is not an element of the set (\ref{holder}). 
We may assume that $b\in \mathcal B(a^+,(1+c)r)$, as otherwise there is nothing to prove. Let
$\beta=\angle (a^+-b^+,b-b^+)$. We have to show $\cos \beta \leq \sqrt{c}(r-r^*)^{1-\sigma}$. This is clear for
$\cos \beta \leq 0$, so let $\cos \beta > 0$. Following the proof of \cite[Prop. 5]{noll} we put
$R=\frac{r}{2} \left(1+ \sqrt{1+\frac{2c+c^2}{\cos^2\beta}}  \right)$. As in \cite{noll} it now follows that
$\mathcal B(b^++Rd,R)$ contains $b$, which implies $R > R(b^+,d)$. Hence by the choice of $U,V$,
$(r-r^*)^\sigma/(R-r^*) < \tau'$. Substituting the definition of $R$ gives
\begin{align*}
    1 &< (r-r^*)^{-\sigma} \tau' \left(r \left(\frac{1}{2}+\frac{1}{2} \sqrt{1+\frac{2c+c^2}{\cos^2\beta}}\right) -r^*\right) \\
    &= (r-r^*)^{1-\sigma} \tau' \left(\frac{1}{2}+\frac{1}{2} \sqrt{1+\frac{2c+c^2}{\cos^2\beta}}\right) + (r-r^*)^{-\sigma}\tau'
    r^* \left(1-\left(\frac{1}{2}+\frac{1}{2} \sqrt{1+\frac{2c+c^2}{\cos^2\beta}}\right)\right) \\
    &\leq  (r-r^*)^{1-\sigma} \tau'  \left(\frac{1}{2}+\frac{1}{2} \sqrt{1+\frac{2c+c^2}{\cos^2\beta}}\right),
\end{align*}
the rightmost term being $\leq 0$. Now suppose that
$\cos\beta > \sqrt{c}(r-r^*)^{1-\sigma}$ contrary to what is claimed, then
\begin{align*}
    1 &< (r-r^*)^{1-\sigma} \tau' \left(\frac{1}{2} + \frac{1}{2} \sqrt{1+\frac{2c+c^2}{c(r-r^*)^{2(1-\sigma)}}}   \right)\\
    &= \frac{\tau'}{2} \left((r-r^*)^{1-\sigma}+\sqrt{(r-r^*)^{2(1-\sigma)}+2+c}  \right)\\
    &< \frac{\tau'}{2} \left(\epsilon + \sqrt{\epsilon^2+2+c}   \right) < 1,
\end{align*}
a contradiction, which proves the result.
\hfill $\square$
\end{proof}

For the following recall the definition of prox-regularity
e.g. in \cite{rock}.

\begin{corollary}
\label{non_shrinking}
Let $(A^*,B^*,r^*)$ be a gap and
suppose $B$ is prox-regular at the points of $B^*$ with reach $>r^*$. Then for every constant $c > 0$ and
every $\sigma \in (0,1)$ the gap is H\"older regular with constant $c$ and exponent $\sigma$.
\end{corollary}

\begin{proof}
Here $\tau = 0$, so for every $\tau' > 0$ and $\sigma' \in (0,1)$ the set $B$ has $(1-\sigma')$-slowly shrinking
reach with rate $\tau'$ for the gap $r^*$ for any constant $c$ with $\frac{\tau'}{2} \sqrt{2+c} < 1$. Given any $c>0$, we can adjust
$\tau'\ll 1$ so that this condition is met, and we let $\sigma = 1-\sigma'$.
\hfill $\square$
\end{proof}

Applying the argument of  Proposition \ref{prox} to prox-building blocks gives the following
extension of \cite[Cor. 3]{noll}. 

\begin{corollary}
\label{cor3}
Consider a prox-alternating sequence $a_{k-1}\to b_{k-1} \stackrel{p}{\to} a_k$ with gap $(A^*,B^*,r^*)$. 
Suppose $B$ is prox-regular with reach $>r^*$  at the points of $B^*$. Then for every constant $c>0$ and every $\sigma\in (0,1)$ the gap
is H\"older regular with constant $c$ and exponent $\sigma$ for the sets $A^s,B^s$. 
\hfill $\square$
\end{corollary}

\begin{example}
Let $B=\{(x,|x|^{3/2}): x\in \mathbb R\}$,
then $B$ has vanishing reach at the origin in direction $d=(0,1)$. We claim that the radius $R_x$ of the largest ball
touching $B$ at $b=(x,|x|^{3/2})$ from above is of the order $R_x=O(|x|^{1/2})$ as $x\to 0$. This can be seen
as follows. An upper bound for $R_x$ is the radius of the osculating circle at $(x,|x|^{3/2})$, which is $\overline{R}_x=\frac{4}{3}|x|^{1/2}(1+\frac{9}{4}|x|)^{3/2}$,
so for small $x$ we have $\overline{R}_x\sim\frac{4}{3}|x|^{1/2}$. For a lower bound, note that for
a plane $C^2$-curve with positive reach and without bottlenecks the reach is
$1/\sigma$ when 
$\sigma$ is  the maximal curvature, cf. \cite{aamari}. To apply this we approximate $B$ by curves $B_\epsilon$ with positive reach. We let
$y=ax^2+bx+c$ on $(-\infty,\epsilon]$ and $y=x^{3/2}$ on $[\epsilon,\infty)$  so that the combined function is $C^2$. This works with
$a=\frac{3}{8}\epsilon^{-1/2}$, $b=\frac{3}{4}\epsilon^{1/2}$, $c=-\frac{1}{8}\epsilon^{3/2}$. Now the reach $r_\epsilon$ of $B_\epsilon$ 
can be computed exactly via \cite{aamari} and is bounded below by
$r_\epsilon\geq \frac{3}{4}\epsilon^{-1/2}$. This means any ball touching $B_\epsilon$ from above with radius
$r< r_\epsilon$ has a unique contact point. Since this is also true for the contact points $b=(x,x^{3/2})\in B$ with $x > \epsilon$, we see that 
the reach $R_x$ of $b=(x,x^{3/2})$ with $x \geq \epsilon$ is $\geq O(\epsilon^{1/2})$. Namely
$\frac{3}{4} |x|^{1/2} \leq R_x \leq \frac{4}{3}|x|^{1/2}$ as $x\to 0$, proving $R(b,d)=O(|x|^{1/2})$ for the
denominator in (\ref{slow}).

Now let $1< \alpha
< \frac{3}{2}$ and put $A=\{(x,|x|^\alpha): x\in \mathbb R\}$, so that $A$ is above $B$ and touches it at the origin. 
Let $a=(y,y^\alpha)$, $b= (x,x^{3/2})$, $b=P_B(a)$, then
the ansatz $(x,x^{3/2})+t(-\frac{3}{2}x^{1/2},1)=(y,y^\alpha)$ gives $t=y^\alpha -x^{3/2}$ and, 
$y(1+\frac{3}{2}x^{1/2}y^{\alpha-1})=x(1+\frac{3}{2}x)$, hence
$\|a-b\| = |y^\alpha - x^{3/2}|\sqrt{1+\frac{9}{4}x}\sim |y^\alpha - x^{3/2}|
\sim |x^\alpha \left(\frac{1+({3}/{2})x}{1+(3/2)x^{1/2}y^{\alpha-1}}\right)^\alpha-x^{3/2}|
= x^\alpha (1+o(1)-x^{3/2-\alpha}) \sim x^\alpha$.
Then $\frac{\|a-b\|^\sigma}{R(b,d)} = O(x^{\alpha\sigma-\frac{1}{2}})$, which is $O(1)$ for $\sigma \geq 1/2\alpha$, so that $B$ has $\sigma$-slowly vanishing reach with respect to
$A$.

For the infeasible case we use
$B = \{(x,|x|^{3/2}+\frac{1}{2}x^2): x\in \mathbb R\}$, then $B$ has slowly shrinking reach at $(0,0)$ with regard to
$A=\{(x,|x|^\alpha +1): x\in \mathbb R\}$ and
gap
value $r^*=1$.
\end{example}

\section{Three-point estimate}
\label{sect_three}
The following result extends \cite[Lemma 1]{noll}, where it was given for
the feasible case $r^*=0$.

\begin{lemma}
{\bf (Three-point estimate)}.
\label{three-point}
Suppose the building block $b \to a^+\to b^+$ satisfies the angle
condition for $r^*$ with constant $\gamma > 0$ and exponent $\omega$.
Suppose further that the building block is $\omega/2$-H\"older regular
with constant $c>0$ satisfying $c < \gamma/2$. Then it satisfies the three-point estimate
\begin{equation}
\label{three_point_estimate}
    \|a^+-b^+\|^2 + \ell \|b-b^+\|^2 \leq \|b-a^+\|^2
\end{equation}
with $\ell = \min\left\{\frac{1}{2},1-\sqrt{\frac{2c}{\gamma}},\frac{c}{2+c}\right\}$ depending only on $c,\gamma$.
\end{lemma}

\begin{proof}
Following the proof of \cite[Lemma 1]{noll} we have to
show that $\frac{1-\ell}{2}\|b-b^+\|\geq \|a^+-b^+\|\cos\beta$, where $\beta = \angle(b-b^+,a^+-b^+)$.
As in that reference there are three cases.  Case I is when $\beta\in [\pi/2,\pi]$, where
$\ell=1/2$ works.  Case II is when $\beta \in [0,\frac{\pi}{2})$, and the latter has two subcases IIa and IIb.

Case IIa is when $b\in \mathcal B(a^+,(1+c)r)$, in which event regularity gives
$\cos\beta \leq \sqrt{c} (r-r^*)^{\omega/2}$. Here we need the angle condition. With $\alpha = \angle(b-a^+,b^+-a^+)$
the cosine theorem gives
\begin{align*}
\|b-b^+\|^2 &\geq 2\|b-a^+\| \|a^+-b^+\|(1-\cos\alpha)\\
&\geq 2\gamma \|b-a^+\|\|a^+-b^+\| (r-r^*)^\omega\\
& \geq \frac{2\gamma}{c}\|a^+-b^+\|^2 \cos^2\beta.
\end{align*}
This leads to $\ell=1-\sqrt{\frac{2\gamma}{c}}$.

The remaining  case IIb is when $\cos\beta > \sqrt{c}(r-r^*)^{\omega/2}$. Here by H\"older regularity
we must have $\|b-a^+\|\geq (1+c)r$. 
Now the argument in part 4)  of \cite[Lemma 1]{noll} can be adopted without changes and requires $\ell = \frac{c}{c+2}$.
Altogether, covering the three cases gives the formula for $\ell $ in the statement.
\hfill $\square$
\end{proof}

\begin{remark}
The three point inequality could be considered as stand-alone, as in \cite{csiszar,chinese}.
Here, following \cite{noll}, we consider it as a technical tool, to be derived from regularity of the sets, as this includes the important case
when $B$ is prox-regular.  
\end{remark}

\section{Convergence}
\label{sect_convergence}
In the feasible case \cite{noll} local convergence was understood in the sense that
if an alternating sequence gets sufficiently close to $A \cap B$, then it converges to some point
in the intersection. Presently we obtain a similar statements for gaps $(A^*,B^*,r^*)$. 
If the $a_k$ get close to $A^*$ and the $b_k$   close
to $B^*$, and if  $r^* <\|a_k-b_k\| < r^* +\eta$ for some small $\eta > 0$, then we expect convergence $a_k \to a^*\in A$, $b_k \to b^*\in B$
to a pair $\|a^*-b^*\|=r^*$  realizing the gap. 
As in the feasible case, this requires the angle condition in tandem with H\"older regularity. However, for $r^* >0$
we need a third ingredient.

\begin{definition}
We say that a gap $(A^*,B^*,r^*)$ is saturated, if for every neighborhood $V$ of $A^*$
there exists a neighborhood $U$ of $B^*$ such that $P_A(b)\subset V$ for every $b\in B \cap U$.
\end{definition}

\begin{remark}
Note that every zero gap $(F,F,0)$ with $F \subset A \cap B$ is saturated. Indeed, let $V=\mathcal N(F,\delta)$
be the neighborhood of $F$, and choose $U=V$. If $b\in U \cap B$, then
there exists $c\in F$ with $\|b-c\| < \delta$. But $F \subset A$, hence $c\in A$, hence $d_A(b) \leq \|b-c\| < \delta$ gives
$P_A(b) \subset \mathcal N(F,\delta)$.
\end{remark}

\begin{remark}
If $P_A$ is single-valued on the set $\{b\in B\cap \mathcal N(B^*,\delta): r^* < {\rm dist}(b,A^*) < r^*+\eta\}$, then the gap is automatically saturated. 
This could still allow $P_A$ to be many-valued on $B^*$. 
\end{remark}

\begin{remark}
The gap $(A^*,B^*,r^*)$ of accumulation points of an alternating sequence $a_k,b_k$ is automatically
saturated with regard to the  underlying sets $A^s=\{a_k: k\in \mathbb N\} \cup A^*$ and
$B^s=\{b_k: k\in \mathbb N\} \cup B^*$. This remains true for a prox-alternating sequence $a_{k-1}\to b_{k-1}\stackrel{p}{\to} a_k$.
\end{remark}

\begin{definition}
We say that the alternating sequence reaches the $\delta$-neighborhood of the gap 
$(A^*,B^*,r^*)$ if there exists $k\in \mathbb N$ with $b_{k}\in \mathcal N(B^*,\delta)$, 
$a_{k+1}\in \mathcal N(A^*,\delta)$, $r^*< \|a_k-b_k\|< r^*+\delta$.
\end{definition}

\begin{theorem}
\label{theorem1}
{\bf (Local attraction).}
Suppose $B$ satisfies the angle condition with exponent $\omega = 4\theta-2$,  $\theta \in [\frac{1}{2},1)$ and constant $\gamma > 0$
for the saturated gap $(A^*,B^*,r^*)$. 
Moreover, suppose the gap is $\omega/2$-H\"older regular with constant $c < \frac{\gamma}{2}$. Then there exists $\delta > 0$
such that whenever an alternating sequence 
reaches the $\delta$-neighborhood of the gap, then
$b_k \to b$ for some $b\in B$ realizing the gap $r^*$. If $A^\sharp$ is the set of accumulation points of the
$a_k$, then $d_{A^\sharp}(b)=d_A(b)=r^*$. 
\end{theorem}

\begin{proof}
1)
Since there is nothing to prove if the iterates attain the gap in a finite number of steps, we assume that  the sequence $b_k$ is infinite. 
By Lemma \ref{three-point} there exists a neighborhood $\mathcal N(A^*,\epsilon)$ of $A^*$, $\eta > 0$, and $\ell \in (0,1)$, such that for
every building block $b_k \to a_{k+1} \to b_{k+1}$ with $a_{k+1}\in V=\mathcal N(A^*,\epsilon)$ and $r^*<\|a_{k+1}-b_{k+1}\|< r^*+\eta$
the three-point-estimate 
\[
\|b_k-a_{k+1}\|^2 \geq \|a_{k+1}-b_{k+1}\|^2 +\ell \|b_k-b_{k+1}\|^2
\]
is satisfied.
Then for these $a_{k+1}\in \mathcal N(A^*,\epsilon)$
we have also the four point estimate
\begin{equation}
\label{four}
d_B(a_k)^2 - d_B(a_{k+1})^2 \geq \ell \|b_k-b_{k+1}\|^2.
\end{equation}

2) Since the neighborhood $V=\mathcal N(A^*,\epsilon)$ chosen in Lemma \ref{three-point} has the property that
the angle inequality is satisfied for the gap $(A^*,B^*,r^*)$
as soon as $a_k\in V$,  we have
\[
\frac{1-\cos \alpha_k}{\left( \|a_k-b_k\|-r^{*}   \right)^{4\theta-2}} \geq \gamma,
\]
where $\alpha_k=\angle( b_{k-1}-a_k,b_k-a_k)$.
Now following the lead of \cite[Thm. 1]{noll} we apply the cosine theorem to obtain
\begin{align}
\label{est1}
\begin{split}
    \|b_{k-1}-b_k\|^2 &= \left(\|b_{k-1}-a_k\|-\|a_k-b_k\|  \right)^2 + 2 \|b_{k-1}-a_k\| \|a_k-b_k\| (1-\cos \alpha_k) \\
    &\geq 2 \gamma \|b_{k-1}-a_k\| \|a_k-b_k\| \left( \|a_k-b_k\|- r^{*}   \right)^{4\theta-2} \\
    & \geq 2 \gamma d_B(a_k)^2 \left(  d_B(a_k)-r^{*}\right)^{4\theta-2}.
    \end{split}
\end{align}
Here we have dropped the square term on the right and used $\|b_{k-1}-a_k\|\geq \|a_k-b_k\|=d_B(a_k)$.
Taking square roots and re-arranging gives 
\begin{equation}
\label{cosine}
    \left(d_B(a_k)-r^{*} \right)^{-2\theta+1} \geq \sqrt{2\gamma} d_B(a_k) \|b_{k-1}-b_k\|^{-1}.
\end{equation}
At this point we observe a junction, because when $r^*=0$, the term $d_B(a_k)$ on the right
matters and leads to the estimate $d_B(a_k)^{-2\theta} \geq \sqrt{2\gamma} \|b_{k-1} - b_k\|^{-1}$.
This case was handled in \cite[Theorem 1]{noll}, so we may for the moment concentrate on the infeasible case $r^* >0$.
The difference will be relevant in the next theorem when rates of convergence will be computed.

Here we use the fact that $s \mapsto s^{2-2\theta}$ is concave, so that
$s_1^{2-2\theta}  -s_2^{2-2\theta} \geq (2-2\theta) s_1^{1-2\theta} (s_1-s_2)$. We apply this to
$s_1 = d_B(a_k)-r^{*}$ and $s_2=d_B(a_{k+1})-r^{*}$ to obtain
\begin{align}
\label{est2}
\begin{split}
    \left[d_B(a_k)-r^{*}  \right]^{2-2\theta} - &\left[ d_B(a_{k+1})-r^{*}  \right]^{2-2\theta} \geq  \\
    &\geq (2-2\theta) \left[d_B(a_k)-r^{*}  \right]^{1-2\theta} \left( d_B(a_k)-d_B(a_{k+1})   \right)\\
    &\geq (2-2\theta)2^{1/2}\gamma^{1/2}d_B(a_k) \|b_{k-1}-b_k\|^{-1} \left( d_B(a_k)-d_B(a_{k+1})   \right)\\
    & \geq(2-2\theta) 2^{1/2}\gamma^{1/2}\ell  \frac{d_B(a_k)}{d_B(a_k)+d_B(a_{k+1})} \|b_{k-1}-b_k\|^{-1} \|b_k - b_{k+1}\|^2 \\
    & \geq(2-2\theta) 2^{-1/2}\gamma^{1/2}\ell   \|b_{k-1}-b_k\|^{-1} \|b_k - b_{k+1}\|^2,
    \end{split}
\end{align}
where the next to last line is obtained by applying (\ref{four}).
Multiplying by $\|b_{k-1}-b_k\|$ and putting $C=(2-2\theta)^{-1} 2^{1/2}\gamma^{-1/2}\ell^{-1}$, we get
\[
C\left( \left[d_B(a_k)-r^{*}  \right]^{2-2\theta} - \left[ d_B(a_{k+1})-r^{*}  \right]^{2-2\theta} \right) \|b_{k-1}-b_k\|
\geq \|b_k-b_{k+1}\|^2.
\]
using the fact that $a^2\leq bc$ implies $a\leq \frac{1}{2}b+\frac{1}{2}c$ for positive $a,b,c$, we deduce
\begin{equation}
\label{sole}
\|b_k-b_{k+1}\| \leq \frac{1}{2} \|b_k-b_{k-1}\| + \frac{C}{2}\left( \left[d_B(a_k)-r^{*}  \right]^{2-2\theta} - \left[ d_B(a_{k+1})-r^{*}  \right]^{2-2\theta} \right).
\end{equation}
Altogether, what we have proved in 1), 2) above is that $a_k,a_{k+1} \in V=\mathcal N(A^*,\epsilon)$
implies (\ref{sole}).

3)
Let us now define our $\delta > 0$. First choose $\delta > 0$ with $\delta < \eta$, $\delta < \epsilon/3$ and
$\delta'' := \sqrt{\delta (2r^*+\delta)/\ell} < \epsilon/6$,
and such that $a\in \mathcal N(A^*,\delta)$ implies $\delta' := \frac{C}{2} (d_B(a) - r^*)^{2-2\theta} < \frac{\epsilon}{3}$.
The latter is possible, since $d_B$ has constant value $r^*$ on $A^*$, so that $\delta'\to 0$ as $a$ gets closer to $A^*$, or what is the same,
$\delta' \to 0$ as $\delta \to 0$. Using this and the fact 
that the gap is saturated, shrink $\delta > 0$ further such that $P_A(B \cap \mathcal N(B^*,\delta+\delta'+2\delta'')) \subset \mathcal N(A^*,\epsilon)$. 
Note that we have
$\delta + \delta' + 2\delta'' < \epsilon$.
We show that $\delta$ is as claimed in the statement.

Relabeling the sequence, we may assume that we have reached
\begin{equation}
    \label{reached1}
b_0 \in \mathcal N(B^*,\delta),   \;  
\; a_1 \in \mathcal N(A^*,\delta), \; r^* <\|a_0-b_0\| < r^*+\delta.
\end{equation}
From this we first deduce $\|b_0-b_1\| < \delta''$. Indeed, from the three point estimate and (\ref{reached1})
we get $r^{*2} <\|a_0-b_0\|^2 + \ell \|b_0-b_1\|^2 \leq \|a_1-b_0\|^2 \leq (r^*+\delta)^2$, hence
$\ell \|b_0-b_1\|^2 \leq (r^*+\delta)^2-r^{*2} = \delta(2r^*+\delta)$, hence
$\|b_0-b_1\|\leq \sqrt{\delta(2r^*+\delta)/\ell} = \delta'' < \epsilon/6$. That gives
\begin{equation}
    \label{reached2}
    b_1  \in \mathcal N(B^*,\delta+\delta''), \quad
a_2 \in \mathcal N(A^*,\epsilon),
\end{equation}
the latter using $a_2\in P_A(b_1)$,  $b_1 \in B\cap \mathcal N(B^*,\delta+\delta'' )\subset B \cap \mathcal N(B^*,\delta+\delta'+\delta'')$, and
the fact that we assured  above that
$P_B(B \cap \mathcal N(B^*,\delta+\delta'+\delta'')) \subset \mathcal N(A^*,\epsilon)$.

4) We will now prove the following
two conditions by induction over $k\geq 1$:
\begin{equation}
\label{16}
    b_{k-1},b_{k} \in \mathcal N(B^*,\epsilon),\quad  a_k, a_{k+1} \in \mathcal N(A^*,\epsilon) 
\end{equation}
and
\begin{equation}
\label{17}
    \sum_{j=1}^{k} \|b_j-b_{j+1}\| \leq \frac{1}{2} \sum_{j=1}^{k} \|b_{j-1}-b_j\| + \frac{C}{2} \left([d_B(a_1)-r^{*}]^{2-2\theta}-[d_B(a_{k+1})-r^{*}]^{2-2\theta}   \right).
\end{equation}

Let us initialize the induction.
We prove $(\ref{16})_1$. Since the sequence has reached the neighborhood of the gap, we have 
(\ref{reached1}), (\ref{reached2}), and
since $\delta+\delta'' < \epsilon$, condition $(\ref{16})_1$ is clear. Now to prove $(\ref{17})_1$, since
$a_1,a_2\in \mathcal N(A^*,\epsilon)$ by $(\ref{16})_1$ just proved,   we get from part 1)-2) that
\[
\|b_1-b_2\| \leq \frac{1}{2} \|b_0-b_1\| + \frac{C}{2} \left([d_B(a_1)-r^{*}]^{2-2\theta}-[d_B(a_{2})-r^{*}]^{2-2\theta}   \right),
\]
which is just $(\ref{17})_1$. This settles initialization.

Let us now do the induction step.
Suppose $(\ref{16})_{k-1}$, $(\ref{17})_{k-1}$ are satisfied for some $k\geq 2$. We have to prove them
for $k$.  Let us first show that
$(\ref{16})_{k-1} \wedge (\ref{17})_{k-1} \implies (\ref{16})_k$. Indeed, from $(\ref{17})_{k-1}$ we get
\[
\sum_{j=1}^{k-1}\|b_j-b_{j+1}\| \leq \frac{1}{2} \sum_{j=1}^{k-1}\|b_{j-1}-b_j\| +  \frac{C}{2} \left([d_B(a_1)-r^{*}]^{2-2\theta}-[d_B(a_{k})-r^{*}]^{2-2\theta}   \right)
\]
hence
\begin{align}
\label{seefrom}
\sum_{j=1}^{k-1}\|b_j-b_{j+1}\|& \leq \|b_0-b_1\| + C \left([d_B(a_1)-r^{*}]^{2-2\theta}-[d_B(a_{k})-r^{*}]^{2-2\theta}\right)  - 2\|b_{k-1}-b_{k}\| \notag\\
&\leq \|b_0-b_1\| +  C [d_B(a_1)-r^{*}]^{2-2\theta}.
\end{align}
Therefore, if we fix $b^* \in B^*$ such that $\|b_1-b^*\| < \delta+\delta''$, then
\begin{align*}
    \|b_{k}-b^*\| &\leq \|b_{k}-b_1\| + \|b_1 - b^*\| \leq \sum_{j=1}^{k-1} \|b_j-b_{j+1}\| + \|b_1-b^*\|  \\
    &\leq \|b_0-b_1\| +  C [d_B(a_1)-r^{*}]^{2-2\theta} + \|b_1 - b^*\|\\
    & < \delta'' + \delta' + \delta + \delta'' < \epsilon,
\end{align*}
using (\ref{seefrom}) and $\delta+\delta'+2\delta'' < \epsilon$, 
so we are done for $b_{k}$. Now since $a_{k+1}\in P_A(b_k)$ and $b_k \in \mathcal N(B^*,\delta + \delta'+2\delta'')$, it also
follows that $a_{k+1} \in \mathcal N(A^*,\epsilon)$, because $P_A(B \cap \mathcal N(B^*,\delta+\delta'+2\delta''))\subset \mathcal N(A^*,\epsilon)$. 
This settles $(\ref{16})_k$.

Now by $(\ref{16})_{k}$ we have $a_k,a_{k+1} \in \mathcal N(A^*,\epsilon)$, hence the argument of
1) 2) gives us $(\ref{sole})_k$. But adding $(\ref{sole})_k$ and $(\ref{17})_{k-1}$ gives $(\ref{17})_k$.
That ends the induction step.

5) To conclude, as $(\ref{17})$ is now true for all $k$, we see e.g. from (\ref{seefrom})
that the series $\sum_{j=1}^\infty\|b_j-b_{j+1}\|$ converges, hence $b_k$ is a Cauchy sequence, which converges to some
$b^\sharp\in B$. But then (\ref{est1}) shows $d_B(a_k)-r^*\to 0$, so every accumulation point $a^\sharp\in A$  of the $a_k$ satisfies $\|b^\sharp -a^\sharp\|=r^*$,
and the gap $r^*$ is realized.
\hfill $\square$
\end{proof}

\begin{remark}
Saturatedness is used in (\ref{reached2}) to assure
that when $b_0,a_1$ have reached the neighborhood of the gap, the next iterate  $a_2$ stays close. For individual sequences
approaching their own gap this is automatically true, but  local attraction has to work simultaneously for {\em all} sequences
getting close to a given gap. Saturatedness is also redundant when $P_A$ is single valued at $b_1$, or when $r^*=0$, as shown in \cite{noll}.
\end{remark}

\begin{remark}
We stress that it is not claimed that $b^\sharp \in B^*$, nor do we have $A^\sharp \subset A^*$. This was already observed in \cite{noll} for
the feasible case. Observe a difference between the case $r^* > 0$ and the zero gap case. With $r^* > 0$ we do not readily 
obtain convergence of the $a_k$, while this holds when $r^*=0$. On the other hand we see that
$\sum_{k=1}^\infty \left(d_B(a_k) - r^* \right)^{2\theta-1} < \infty$. 
\end{remark}

\begin{theorem}
\label{rate}
{\bf (Rate of convergence for $r^*>0$).}
Under the hypotheses of Theorem {\rm \ref{theorem1}}, for $r^* >0$ the speed of convergence is
$\|b_k-b^*\| = O(k^{-\frac{1-\theta}{2\theta-3/2}})$ for  $\theta\in (\frac{3}{4},1)$. For $\theta=\frac{3}{4}$
convergence is R-linear.
For 
$\theta \in (\frac{1}{2},\frac{3}{4})$
convergence is R-linear with rate $\frac{1}{2}+\epsilon$, where $\epsilon > 0$
can be chosen arbitrarily small. For $\theta=\frac{1}{2}$ convergence is finite.
\end{theorem}

\begin{proof}
1)
Summing (\ref{sole}) from $k=N$ to $k=M$ for $M > N$ gives
\[
-\frac{1}{2}\|b_N-b_{N-1}\| + \frac{1}{2}\sum_{k=N}^{M-1}\|b_k-b_{k+1}\| + \|b_M-b_{M+1}\|
\leq \frac{C}{2} \left(\left[ d_B(a_N)-r^{*} \right]^{2-2\theta} -  \left[d_B(a_{M+1}-r^{*} \right]^{2-2\theta} \right),
\]
and passing to the limit $M\to \infty$ leads to
\[
-\frac{1}{2}\|b_N-b_{N-1}\| + \frac{1}{2}\sum_{k=N}^{\infty}\|b_k-b_{k+1}\| \leq \frac{C}{2} \left[ d_B(a_N)-r^{*} \right]^{2-2\theta}.
\]
Introducing $S_N=\sum_{k=N}^{\infty}\|b_k-b_{k+1}\|$, we
have 
\begin{equation}
\label{all_the_way}
-\frac{1}{2} \left(S_{N-1}-S_N   \right) + \frac{1}{2} S_N \leq \frac{C}{2} \left[ d_B(a_N)-r^{*} \right]^{2-2\theta}.
\end{equation}
Now from (\ref{est1})
\begin{equation}
\label{this}
\left(d_B(a_N)-r^{*}\right)^{2-2\theta} \leq \left(\sqrt{2\gamma}\, d_B(a_N)  \right)^{\frac{2-2\theta}{1-2\theta}} \|b_{N-1}-b_N\|^{\frac{2-2\theta}{2\theta-1}}
\end{equation}
hence regrouping and using $d_B(a_N) < r^*+\eta$ gives
\begin{equation}
\label{use_again}
 \frac{1}{2}S_N \leq \frac{C}{2} \left(\sqrt{2\gamma}\, (r^*+\eta)  \right)^{\frac{2-2\theta}{1-2\theta}} (S_{N-1}-S_N)^{\frac{2-2\theta}{2\theta-1}}
 + \frac{1}{2}(S_{N-1}-S_N).
\end{equation}

2)
Now consider $\theta \in (\frac{3}{4},1)$, then the term $(S_{N-1}-S_N)^{\frac{2-2\theta}{2\theta-1}}$ ultimately
dominates $S_{N-1}-S_N$, so there exists a constant $C'>0$ such that for large enough $N$,
\[
S_N^{\frac{2\theta-1}{2-2\theta}} \leq C' (S_{N-1}-S_N).
\]
From here onward the proof follows exactly the line in \cite[Cor. 4]{noll}, and we arrive
at the estimate $S_M = O(M^{-\frac{2-2\theta}{4\theta-3}})$, which implies
$\|b_M-b^*\| = O(M^{-\frac{2-2\theta}{4\theta-3}})$ as $M\to\infty$.

 3)  For $\theta = \frac{3}{4}$ the estimate (\ref{use_again}) gives
$S_N \leq (1+C'')(S_{N-1}-S_N)$ with $C''=C\left(  \sqrt{2\gamma} (r^*+\eta) \right)^{-1}$. Hence
$S_N \leq \frac{1+C''}{2+C''} S_{N-1}$ gives Q-linear convergence $S_N \to 0$, hence R-linear convergence $b_N-b^*\to 0$.

4)
For $\theta \in (\frac{1}{2}, \frac{3}{4})$ the term $(S_{N-1}-S_N)^{\frac{2-2\theta}{2\theta-1}}$ is dominated by $S_{N-1}-S_N$, hence
we get $S_N \leq (1+C'')(S_{N-1}-S_N)$ with a constant $C''$ that can be made arbitrarily small.
Then $S_N \leq \frac{1+C''}{2+C''} S_{N-1}$ with a Q-linear rate that can be chosen arbitrarily close to $\frac{1}{2}$. 

5)
Finally, for $\theta =\frac{1}{2}$ 
the angle condition gives $1-\cos \alpha_k > \gamma > 0$, so the angles $\alpha_k=\angle (b_{k-1}-a_k,b_k-a_k)$ stay away from $0$. 
Since $r^* > 0$, this means $\|b_{k-1}-b_k\| \geq 2 \sin (\alpha_k/2) r^* > \epsilon > 0$. But since we proved in Theorem \ref{theorem1} that the sequence $b_k$
converges when it is infinite, we conclude that
the sequence
$b_k$ must converge finitely. 

\hfill $\square$
\end{proof}

\begin{corollary}
\label{old}
{\bf (Rate of convergence for $r^*=0$, cf. \cite{noll})}
Under the hypotheses of Theorem {\rm \ref{theorem1}}, now with $r^*=0$, the speed of convergence
is $\|b_k-b^*\| =O(k^{-\frac{1-\theta}{2\theta-1}})$, $\|a_k-b^*\|=O(k^{-\frac{1-\theta}{2\theta-1}})$ for $\theta\in (\frac{1}{2},1)$. For $\theta=\frac{1}{2}$ the speed is R-linear.
\end{corollary}

\begin{proof}
We can go all the way till (\ref{all_the_way}) in the above proof. But now due to $r^*=0$, (\ref{est1}) reads
$d_B(a_k)^{-2\theta} \geq \sqrt{2\gamma} \|b_{k-1}-b_k\|^{-1}$. Then we get the estimate
$d_B(a_k)^{2-2\theta} \leq (2\gamma)^{\frac{\theta-1}{2\theta}} \|b_{k-1}-b_k\|^{\frac{1-\theta}{\theta}}$
replacing (\ref{this}). As seen in \cite[Cor. 4]{noll}, this leads to the slightly slower rate
$\|b_k-b^*\|=O(k^{-\frac{1-\theta}{2\theta-1}})$, which due to $r^*=0$ then also holds for the $a_k$. 
\hfill$\square$
\end{proof}

\begin{remark}
Consider $\phi(x)=1+\frac{1}{2}x^2$ and $B={\rm epi}(\phi)\subset \mathbb R^2$, $A$ the $x$-axis. Then convergence of alternating projections
to the gap $(\{(0,0)\},\{(0,1)\},1)$ is with asymptotic linear rate $\frac{1}{2}$.  With some more elementary calculus one can show that 
$f = i_A + \frac{1}{2}(d_B-1)^2$ has \L ojasiewicz exponent $\theta = \frac{3}{4}$, which corroborates the statement of Theorem \ref{rate} for that case.

If we shift the set $B$ down by letting $\psi(x)=\frac{1}{2}x^2$, $B={\rm epi}(\psi)$, so that $A,B$ touch at the origin,
then even though $f = i_A + \frac{1}{2} d_B^2$ still has \L ojasiewicz exponent $\theta=\frac{3}{4}$, this now in accordance with
Corollary \ref{old}
only assures a sublinear rate $O(k^{-1/2})$. 
Since $A,B$ are convex, this is not surprising, as here linear convergence would require $A,B$ to intersect {\it at an angle}, and not
tangentially. For the \L ojasiewicz exponent of $i_A + \frac{1}{2}d_B^2$ in the convex case see also \cite{bolte_new},
and for general considerations as to obtaining optimal $\theta$ see \cite{feenham}.
\end{remark}

\begin{corollary}
\label{global}
{\bf (Global convergence for $r^*>0$).} Let $a_k,b_k$ be a bounded prox-alternating sequence with gap $\|a_k-b_k\|\to r^* >0$.
Suppose $B^s$ satisfies the angle condition 
for that gap with exponent $\omega = 4\theta -2$,  $\theta \in (\frac{3}{4},1)$ and constant $\gamma > 0$, and suppose the gap is $\omega/2$-H\"older regular
with constant $c< \gamma/2$. Then $b_k \to b^*$ for some $b^*\in B$ with rate $\|b_k-b^*\| = O(k^{-\frac{1-\theta}{2\theta-3/2}})$. 
For $\theta=\frac{3}{4}$ the speed is R-linear.
\end{corollary}

\begin{proof}
Let $A^*,B^*$ be the sets of accumulation points of the sequences $a_k,b_k$, then $(A^*,B^*,r^*)$ is a saturated gap
for $A^s,B^s$.
By hypothesis the set $B^s$ satisfies the angle condition with exponent $\omega = 4\theta-2$,  $\theta\in [\frac{3}{4},1)$ and constant $\gamma > 0$
for this gap, and moreover the gap is $\omega/2$-H\"older regular with constant $c < \gamma/2$.
The alternating sequence therefore automatically reaches the gap, and
the main convergence theorem with the underlying sets $A^s,B^s$  implies convergence of the $b_k$. The speed of convergence 
follows from Theorem \ref{rate}, which in terms of $\omega$ is $O(k^{-\frac{2-\omega}{2\omega-2}})$.
\hfill $\square$
\end{proof}

The corresponding global convergence theorem for the case $r^*=0$ is obtained in the same way using using the sets $A^s,B^s$ and
\cite[Theorem 1]{noll}, which leads to the rate $\|b_k-b^*\| = O(k^{-\frac{1-\theta}{2\theta-1}})=O(k^{-\frac{2-\omega}{2\omega}})$, and also
$\|a_k-b^*\| = O(k^{-\frac{1-\theta}{2\theta-1}})$.

\begin{corollary}
\label{subanalytic}
{\bf (Subanalytic sets).}
Suppose $A,B$ are closed subanalytic sets and  $B$ is prox-regular. Let
$a_k,b_k$ be any bounded prox-alternating sequence approaching its gap with value $r^* < R$, where $R>0$ is the reach of $B$ at the points of $B^*$.  
Then the $b_k$ converge with speed $\|b_k-b^*\| = O(k^{-\rho})$ for some $\rho>0$. 
\end{corollary}

\begin{proof}
Let $A^*,B^*$ be the set of accumulation points of the sequences $a_k,b_k$. 
Since $B$ is prox-regular and $B^*$ is compact, $B$ has positive reach $R>0$ at the points of $B^*$. Then by Corollary \ref{cor3}
every gap with $r^* < R$ is H\"older regular on a neighborhood $V$ of $A^*$. By Proposition \ref{prox} we may also assume that the angle condition is satisfied 
on this neighborhood, and
Lemma \ref{three-point} then gives the three-point inequality on  $V$. 
That means the argument 1)+2) in the proof of Theorem \ref{theorem1} works as long as $a_k,a_{k+1}\in V$.

But  $P_A(b_k) \in V$ from some counter $k$ onward, so that whenever the argument above produces a new
$b_{k+1}$ satisfying (\ref{sole}), we have $a_{k+2} = P_A(b_{k+1})\in V$, so that we can iterate the procedure. Therefore
by the main convergence theorem the sequence $b_k$ converges to a $b^*\in B$ realizing the gap $r^*$. 
The speed of convergence is governed by Theorem \ref{rate}.
\hfill $\square$
\end{proof}

The main convergence theorem derives convergence from the angle condition in tandem with the
four-point estimate (\ref{four}). H\"older regularity is only used to prove the latter, but is not used directly in the proof of
Theorem \ref{theorem1}, and similarly already in \cite{noll}. We therefore have the following

\begin{corollary}
Let $a_k,b_k$ be a bounded alternating sequence between $A,B$ such that building blocks $a_{k-1}\to b_{k-1} \to a_k\to b_k$
satisfy the four-point estimate
{\rm (\ref{four})} with the same $\ell >0$.  Suppose $B$ satisfies the angle condition for the gap generated by the alternating sequence.
Then $b_k \to b^*$ for some $b^*\in B$  with speed $O(k^{-\rho})$ for some $\rho >0$.
\hfill $\square$
\end{corollary}

\begin{remark}
This means we can understand (\ref{four}) as a regularity property replacing convexity,
which in tandem with the angle condition assures convergence
with rate. In particular, for $r^* >0$, an R-linear rate is obtained from (\ref{four}) and the angle condition (\ref{angle1}) with $\omega =1$,
while for $r^*=0$, the R-linear rate occurs under (\ref{four})  and the angle condition with $\omega =0$.  
\end{remark}

Let us look for conditions under which not only the sequence $b_k$, but also
the $a_k$, converge. This is obviously the case when the hypotheses in the main convergence theorem or
in corollaries \ref{global}, \ref{subanalytic} are satisfied symmetrically, and we leave this to the reader.
The following observation is also useful.

\begin{remark}
Let $A,B$ be prox-regular and suppose an alternating sequence $a_k,b_k$ within reach of both sets is generated.
Then  the gap $(A^*,B^*,r^*)$ of accumulation points of the $a_k,b_k$ has the property that
$P_A:B^*\to A^*$ is a bijection and $P_B:B^*\to A^*$ is its inverse. In that case, if one of the sequences converges,
then so does the other. 
\end{remark}

This is for instance used in the following, where we recall from \cite{shiota} that semi-algebraic sets
are those which satisfy  (\ref{subanalytic}) with $\phi_{ij}, \psi_{ij}$ polynomials:

\begin{corollary}
{\bf (See \cite{chinese})}.
Suppose $A,B$ are semi-algebraic sets, and let $a_k,b_k$ be a bounded alternating sequence satisfying the
three point inequality. Suppose there exists $L>0$ such that $\|P_A(b)-P_A(b')\| \leq L \|b-b'\|$ for $b,b' \in B^s$.
Then $b_k \to b^*\in B$ and $a_k \to a^*\in A$, both with rate $O(k^{-\rho})$ for some $\rho > 0$.
\end{corollary}

\begin{proof}
With the three-point estimate satisfied by hypothesis, and with the angle condition
satisfied by Lemma \ref{angle},
we get
a neighborhood $V$ of $A^*$
on which the argument 1)-2) in the proof of Theorem \ref{theorem1} works. For $k$ large enough, we have $P_A(b_k)\in V$, hence
condition (\ref{sole}) can be reproduced, and that gives convergence of the $b_k$. Convergence of the $a_k$
then follows easily with the Lipschitz condition.
\hfill $\square$
\end{proof}

We close this section by considering the averaged projection method. For closed sets $C_1,\dots,C_m$ in $\mathbb R^n$, the method
iterates as follows: Given
the current average $x\in \mathbb R^n$, compute projections $x_i \in P_{C_i}(x)$, and form the new average
$x^+ = \frac{1}{m}(x_1+\dots+x_m)$.

\begin{corollary}
{\bf (Averaged projections)}.
Let $C_1,\dots,C_m$ be subanalytic, and let $x^k$ be a bounded sequence of averaged  projections.
Then the $x^k$ converge to a limit average $x^*$ with rate $\|x^k-x^*\|=O(k^{-\rho})$ for some $\rho > 0$. If 
$(x_1^*,\dots,x_m^*)$ is any of the accumulation points of the projections $(x_1^k,\dots,x_m^k)$ with $x_i^k\in P_{C_i}(x^k)$, then
$\frac{1}{m}(x_1^*+\dots+x_m^*) = x^*$ and $x_i^*\in P_{C_i}(x^*)$. 
\end{corollary}

\begin{proof}
As is well-known, we may interpret the situation as alternating projections between
$A=C_1 \times \dots \times C_m$ and the diagonal $B = \{(x,\dots,x): x \in \mathbb R^n\}$. Both sets are
subanalytic, and $B$ is convex, hence the main convergence theorem gives global
convergence of the $B$ iterates, hence of the $x^k$, at rate $O(k^{-\rho})$. As in the general case, a priori nothing can be said about
convergence of the $(x_1^k,\dots,x_m^k)\in C_1 \times\dots \times C_m$, but any of their accumulation points
realizes the gap value $\sum_{i=1}^m (x_i^*-x^*)^2 = mV(x_1^*,\dots,x_m^*)$, which is $m$ times the biased sample variance. 
\hfill $\square$
\end{proof}

In other words, all accumulation points $(x_1^*,\dots,x_m^*)$ of the projected vector have the same sample mean $x^*$ and the same sample variance. 
Naturally, conditions which assure convergence to a single limit are obtained in much the same way as for the general case. For instance, if $m-1$ of the $m$ projections are single valued at $b^*$, then $A^*$ is singleton. A probabilistic interpretation of this result in terms of the EM-algorithm
will be given attention in section \ref{EM}.

\section{Gerchberg-Saxton}
\label{sect_GS}
In phase retrieval one has to determine
an unknown signal $\x(t)$ with physical
coordinates $t=0,\dots,N-1$ from measurements
$|\widehat{\x}(\omega)|^2=m(\omega)^2$ of its Fourier magnitude
obtained
at frequency coordinates
$\omega=0,\dots,N-1$. Given the magnitude $m(\omega)$, we have to recover the
unknown phase
$\widehat{\x}(\omega)/|\widehat{\x}(\omega)|$ of the signal, hence the name. 
As this is generally
an under-determined problem, 
prior information about the unknown $\x(t)$
under the form of a constraint
$\x\in \A$ is added. For instance
in electron microscopy $\x\in \A$ accounts for a second
set of measurements
of the physical domain amplitude or intensity $|\x(t)|^2$, 
while in other situations 
$\x\in \A$ could stand for a pattern like sparsity, 
prior information about the spatial
localization, non-negativity, and much else. An exact solution of the phase retrieval problem
would then be an object $\x\in \mathbb C^N$ with pattern $\x\in \A$ satisfying
$|\widehat{\x}|=m$. Since due to
noisy measurements an exact solution is rarely possible,
the measured data may lead us to accept
pairs $\x^*,\y^*\in \mathbb C^N$
as generalized solutions, where $\y^*$ is a phase retrieval for $\x^*$, and $\x^*$
is a pattern, or prior, for $\y^*$. In other words,
$$ \x^*\in P_\A(\y^*), \; \widehat{\y}^*=m \cdot\widehat{\x}^*/|\widehat{\x}^*|,$$
or in fixed-point terminology:
$$
\x^*=P_\A\left((m\cdot\widehat{\x}^*/|\widehat{\x}^*|)^\sim\right), \; \; \y^* =\big(m\cdot\widehat{P_\A(\y^*)}/|\widehat{P_\A(\y^*)}|\big)^\sim,
$$
where $\sim$ is the inverse Fourier transform.
The Gerchberg-Saxton error reduction method is now the following iterative
procedure:

\begin{algorithm}[!ht]
\caption{\!\!{\bf .} Gerchberg-Saxton error reduction}
\begin{algorithmic}[1]
\STEP{Adjust magnitude} Given current iterate $\x\in \A$, compute Fourier transform $\widehat{\x}$
and correct Fourier magnitude by
computing $\widehat{\y}(\omega)= m(\omega)\cdot\frac{\widehat{\x}(\omega)}{|\widehat{\x}(\omega)|}$.
\STEP{Adjust pattern} Compute inverse Fourier transform $\y$ of $\widehat{\y}$  and
obtain new iterate $\x^+$ as orthogonal projection of
$\y$ on prior information set $\A$, i.e.,  $\x^+\in P_\A({\y})$.
\end{algorithmic}
\end{algorithm}

As is well-known, the magnitude correction step
can be interpreted as orthogonal projection
of the current prior $\x\in \A$ on the magnitude 
set 
\begin{equation}
    \label{fat_B}
\B=\{\y\in \mathbb C^N: |\widehat{\y}(\omega)|=m(\omega),
\omega=0,\dots,N-1\},
\end{equation}
so that Gerchberg-Saxton error reduction
is the special case $\x^+\in P_\A\left( P_\B\left(\x\right)\right)$, $\y^+\in P_\B(P_\A(\y))$
of alternating projections. If we call $\x\in \A$ priors or pattern,
and $\y \in \B$  phase retrievals,
then
a generalized solution of the phase retrieval problem
is a pair $(\x^*,\y^*)$, where $\y^*\in \B$ is a phase retrieval closest to the prior
$\x^*\in \A$, and $\x^*$ is closest to $\y^*$ among the priors.
Since $\B$ is bounded,
the algorithm will by default give a gap
$(\A^*,\B^*,r^*)$, consisting of generalized solutions $(\x^*,\y^*)$. Primarily we hope
$\B^*$  to be singleton, as this means a unique phase retrieval $\y^*$ for all priors $\x^*\in \A^*$.  
Secondarily,  we would also not be averse to
$\A^*$ being singleton, as this would indicate that prior information $\A$ was successful 
in orienting us toward a unique prior $\x^*\in \A$
with that phase retrieval $\y^*\in \B^*$. Notwithstanding, the ideal case is convergence to $\x^*=\y^*\in \A \cap \B$,
in which case we find a prior which is also a phase retrieval.

One may argue that the least useful prior information is
$\A = \{{\bf 0}\}$, as this gives no orientation whatsoever on how to select
a phase retrieval $\y^*\in \B$ among the candidates $\y\in \B$. Any guess $\x\not=0$ would seem better.
We conclude  that meaningful  prior information $\A$ should allow a guess $\x\in \A$
better than just $\x={\bf 0}$. Since dist$({\bf 0},\B) = \|m\|:= \big(\sum_{\omega=0}^{N-1} m(\omega)^2\big)^{1/2}$, we shall 
say that $\A$ {\em allows a prior guess better than} ${\bf 0}$ if there exists
$\x\in \A$ with dist$(\x,\B) < \|m\|$.

\begin{theorem}
\label{GS}
Let prior information $\x\in \A$ be represented
by a closed subanalytic set $\A$ allowing a guess better than ${\bf 0}$.
Suppose Gerchberg-Saxton error reduction is started from that guess and generates sequences
$\x_n \in \A$, $\y_n\in \B$.
Then $\y_n$ converges to  a phase retrieval $\y^*\in \B$ with speed $\|\y_n-\y^*\| = O(n^{-\rho})$ for some $\rho > 0$. Every accumulation point
$\x^*\in \A^*$ of the sequence of priors $\x_n$
has phase retrieval $\y^*=P_\B(\x^*)$, and every prior $\x^*\in \A^*$ is best for $\y^*$, i.e.,  $\A^*\subset P_\A(\y^*)$.
\end{theorem}

\begin{proof}
Since the method is an instance of alternating projections,
the result will
follow from the main theorem. Note that
the sequences $\x_n,\y_n$ generate a gap $(\A^*,\B^*,r^*)$, where
$r^* < \|m\|$, because by assumption the initial guess satisfies already
dist$(\x_0,\B) < \|m\|$.

We check the hypotheses of the main theorem.
The fact that $\B$ is subanalytic was shown in \cite{noll},
and since $\A$ is subanalytic by hypothesis,
the first part of the requirements in the main theorem
is met.

For the following we identify $\mathbb C^N$ with $\mathbb R^{2N}$
in the natural way. Then up to Fourier transforms $P_\B$ is the mapping
\begin{equation}
    \label{project}
(\widehat{\x}_1(\omega),\widehat{\x}_2(\omega))\to m(\omega)
\left( 
\frac{\widehat{\x}_1(\omega)}
     {\sqrt{\widehat{\x}_1(\omega)^2+\widehat{\x}_2(\omega)^2}} ,
     \frac{\widehat{\x}_2(\omega)}
     {\sqrt{\widehat{\x}_1(\omega)^2+\widehat{\x}_2(\omega)^2}} 
\right),
\end{equation}
which can be understood as the cartesian product of
$N$ projections on circles with radii $m(\omega)$
in $\mathbb R^2$.

Working for simplicity in the frequency domain,
let $\y\in \B$ and ${\bf d}$ a unit proximal normal vector to
$\B$ at $\y$. Then ${\bf d}=(d_\omega)$ where for every
$\omega=0,\dots,N-1$ $d_\omega$ is a normal to
the sphere $y_1(\omega)^2+y_2(\omega)^2 = m(\omega)^2$ at $(y_1(\omega),y_2(\omega))
\in\mathbb R^2$. That means $d_\omega=\pm(y_1(\omega),y_2(\omega))/\|m\|$. 
This gives us now the reach of $\B$ at $\y$
with respect to ${\bf d}$. We have
$R(\y,{\bf d})=\|m\|$ if there exists at least one coordinate $\omega$ with $d_\omega=-(y_1(\omega),y_2(\omega))/\|m\|$,
while $R(\y,{\bf d})=\infty$ if all signs are positive. Indeed, 
we have to determine the largest $R\geq 0$ such that 
the projection $P_{\bf B}(\y+R{\bf d})=\y$. We 
may without loss assume that $\y =(0,m(\omega))$, then
$\y+R{\bf d}= (0,m(\omega)) \pm \frac{R}{\|m\|} (0,m(\omega))= (0,(1\pm \frac{R}{\|m\|})m(\omega))$. 
Now it follows that $P_{\bf B}(\y+R{\bf d}) = (0,{\rm sign}(1\pm \frac{R}{\|m\|})m(\omega))$,
and for this to equal $\y = (0,m(\omega))$, we need $1-\frac{R}{\|m\|} > 0$ if there is at least one negative sign,
while this is always true when all signs are positive. This means $R < \|m\|$ if there is one negative sign, so the limiting case
gives the reach $R=\|m\|$.

In consequence, as the sequence $\x_n,\y_n$ has gap $r^* < \|m\|$, the $\x_n\in \A$ are within reach of
$\B$, so H\"older regularity of the gap $(\A^*,\B^*,r^*)$ associated with the alternating sequence
follows from Corollary \ref{non_shrinking}.
Convergence $\y_n\to \y^*\in \B$ with $\B^*=\{\y^*\}$ now follows from the main theorem, the rate being provided by Theorem \ref{rate}.
\hfill $\square$
\end{proof}

\begin{remark}
As in the case of the main theorem no information on convergence of the
sequence $\x_n\in \A$ is available in the infeasible case $r^* > 0$, while convergence of the $\x_n$ is assured \cite{noll} when
$r^*=0$. In the feasible case starting from a guess better than {\bf 0} is not required to get convergence, see \cite{noll}.
Moreover, the projection on $\A$ may be performed locally, which gives additional flexibility.
\end{remark}

When $r^* >0$
additional properties of the prior set $\A$ are needed to
assure that $\A^*$ is also singleton. During the following we discuss a number of prominent examples.
Historically the first instance of Gerchberg-Saxton
error reduction along with (\ref{B}) had measurements of the signal magnitude in a second Fourier plane.
This can be modeled by taking the prior set
\begin{equation}
    \label{second_plane}
\A = \{\x\in \mathbb C^N: |\x(t)| = \widetilde{m}(t), t=0,\dots,N-1\}
\end{equation}
where $\|\widetilde{m}\| = \|m\|$. Here $\B$ and $\A$ have the same reach, and consequently we have the following

\begin{corollary}
\label{original}
The historically first instance of Gerchberg-Saxton error reduction {\rm (\ref{fat_B})}, {\rm (\ref{second_plane})},
if started
from an initial guess $\x_0$ better than ${\bf 0}$, 
converges with speed
$\|\x_k - \x^*\|=O(k^{-\rho})$, $\|\y_k - \y^*\|= O(k^{-\rho})$ for some $\rho > 0$.
The limit pair
$\x^*,\y^*$ has the following properties: $|\x^*|=\widetilde{m}$, $|\widehat{\y}^*|=m$, $\widehat{\y}^* = m\cdot \widehat{\x}^*/|\widehat{\x}^*|$,
$\x^* = \widetilde{m}\cdot {\y}^*/|{\y}^*|$.
\end{corollary}

\begin{proof}
This follows by applying Theorem \ref{GS} to both gaps $(\A^*,\B^*,r^*)$ and $(\B^*,\A^*,r^*)$ and using $r^* < \|m\|=\|\widetilde{m}\|$. 
Note that $\A,\B$ are
both subanalytic, so the hypotheses of the theorem are met.
\hfill $\square$
\end{proof}

\begin{remark}
The case $r^*=0$ is allowed and gives $\x^*=\y^*$. As was shown in \cite{noll}, if $\A \cap \B \not=\emptyset$,
then there exists a neighborhood $V$ of $\A \cap \B$ such that whenever a Gerchberg-Saxton sequence enters $V$,
 it will converge toward a phase retrieval $\x^*=\y^*\in \A \cap \B$.  We mention that the case of two Fourier planes arises for instance in electron microscopy and
 in wave front sensing \cite{fienup,fienup2}.
\end{remark}

Another typical case arising in a variety of applications in crystallography (see \cite{fienup2})  is when the unknown signal $\x$
has support in a known subset $S$ of the physical domain $\{0,\dots,N-1\}$, i.e., supp$(\x) \subset S$.

\begin{corollary}
Consider a support prior
$\A = \{\x\in \mathbb C^N: \x(t) = 0 \mbox{ for } t\not\in S\}$ in the physical domain, where Gerchberg-Saxton error reduction has
the compact form $\x^+ = {\bf 1}_S\cdot (m\widehat{\x}/|\widehat{\x}|)^\sim$,
${\y}^+=\big(m \,\widehat{{\bf 1}_S\cdot \y}/|\widehat{{\bf 1}_S\cdot \y}|\big)^\sim$.
Here the $\x_k\in \A$ converge with speed $O(k^{-\rho})$ to a unique $\x^*$
with its support in $S$, while every accumulation point
$\y^*$
of the $\y_k$ is a possible phase retrieval of $\x^*$.
If, in addition, the prior allows a guess better than ${\bf 0}$ from which the iterates are started, 
then both sequences converge with that speed
to a pair $(\x^*,\y^*)$, where ${\rm supp}(\x^*)\subset S$, $|\widehat{\y}^*|=m$, $\x^*={\bf 1}_S\cdot \y^*$,
$m\widehat{\x}^*/|\widehat{\x}^*|=\widehat{\y}^*$.
\end{corollary}

\begin{proof}
The constraint set $\A$ is convex and algebraic, hence convergence $\x_k\to \x^*\in \A$ follows from
Theorem \ref{GS}, applied to the gap
$(\B^*,\A^*,r^*)$, using that $\A$ has infinite reach. On the other hand, when $r^*< \|m\|$, we can
use the previous result and obtain convergence $\y_k\to \y^*$, so that both sequences converge.
\hfill $\square$
\end{proof}


\begin{remark}
Another case in the rubrique of convex priors is
$\A=\{\x\in \mathbb C^N: {\rm Im}(\x)=0, {\rm Re}(\x)\geq 0\}$, which occurs for instance
if $\x$ is an unknown image, known to have real non-negative gray values. This arises for instance in astronomic speckle interferometry,
cf. \cite{fienup2}. 
\end{remark}

An interesting case often discussed in the 
literature is a sparsity prior. Let $k\ll N$
and define
\begin{equation}
    \label{sparse}
\A = \{\x\in \mathbb C^N: \mbox{at most $k$ of the $\x(t)$ are non-zero}\}.
\end{equation}
The projection $P_{\A}$ on $\A$ is easily identified:
If $|\y(t_0)|\leq |\y(t_1)|\leq \dots \leq |\y(t_{N-1})|$ for a permutation $t_0,\dots,t_{N-1}$ of
$0,\dots,N-1$, then $\x$ with $\x(t_0)=0,\dots,\x(t_{k-1})=0$, $\x(t_k)=\y(t_k),\dots,\x(t_{N-1})=\y(t_{N-1})$
belongs to $P_\A(\y)$, and every element of $P_\A(\y)$ is of this type. Let $S \subset \{0,\dots,N-1\}$ denote subsets of
cardinal $N-k$, and let $P_S$ be the projection on
the linear subspace $\{\x: \x(t) = 0 \mbox{ for }t\not\in S\}$, that is
$P_S(\y) = {\bf 1}_S\cdot \y$.
Then $P_\A(y)\subset \bigcup\{P_S(\y): |S|=N-k\}$. Moreover,
there exists a subset $\frak{S}$ of $\mathcal P(\{0,\dots,N-1\})$, depending on $\y$,
such that $P_\A(\y)=\bigcup \{P_S(\y): S \in \frak{S}\}$.

Now suppose $r^* < \|m\|$, so that the sequence $\y_n$ in the Gerchberg-Saxton algorithm converges to
a unique phase retrieval $\y^*$. Let $\A^*$ be the set of accumulation points
of the sequence $\x_n\in \A$. Then $\A^* = \bigcup\{P_S(\y^*): S \in \frak{S}^*\}$ for the
$\frak{S}^*$ associated with $\y^*$. That means, $\A^*$
is a finite set of cardinal $|\A^*|\leq |\frak{S}^*|\leq {N\choose N-k}$.
$|\A^*|$ can be computed accurately.
Let $|\y^*(t_0)|\leq \dots \leq |\y^*(t_{k-1})| \leq \dots \leq |\y^*(t_{N-1})|$, and suppose
there is ambiguity around position $k-1$ in the sense that
$|\y^*(t_{k-1-r})|=\dots =|\y^*(t_{k-1})|=\dots=|\y^*(t_{k-1+s})|$
for $s>0$. Then we have ${r+s+1 \choose r+1}$ possibilities to arrange $|\y|$
in increasing order and truncate at $k-1$, so this is the cardinal of $\A^*$. Now
choose $\epsilon >0$ such that the balls $\mathcal B(\x^*,\epsilon)$, $\x^*\in \A^*$ are mutually disjoint.
Note that $\x_n\in \bigcup\{ \mathcal B(\x^*,\epsilon): \x^*\in \A^*\}$ from some counter
$n(\epsilon)$ onward. That means we get a finite partition
$\mathbb N= N_1 \cup \dots \cup N_{|\A^*|}$ into infinite sets $N_i$ such that
the subsequence $\x_n$, $n\in N_i$, converges to the $i$th element of $\A^*$.
If $|\y^*(t_{k-1})| < |\y^*(t_k)|$, which corresponds to the case $s=0$, then the projection is unique,
and the entire sequence $\x_n$ converges. 

\begin{corollary}
Suppose $\A$ is the sparsity prior {\rm (\ref{sparse})} and allows a guess
better than ${\bf 0}$, at which Gerchberg-Saxton error reduction is started. Then $\|\y_n - \y^*\|=O(n^{-\rho})$ for a unique phase retrieval $\y^*$, 
while the
$\x_n$ have finitely many sparse accumulation points $\x^*$, each
admitting $\y^*$ as its phase retrieval. If choosing the $k$ smallest $|\y^*(t)|$ is unambiguous,
then the entire sequence $\x_n$ converges to a unique sparse $\x^*$, whose phase retrieval is $\y^*$.
\hfill $\square$
\end{corollary}

\begin{remark}
Due to the special discrete structure of $\A^*$, if it is known that
$\x_n-\x_{n-1} \to 0$, then the sequence $\x_n$ converges as well.
\end{remark}

In \cite{thao} sparsity of the phase in the frequency
domain is considered with the prior
\begin{equation}
\label{sparse_phase}
\A=\{\x\in \mathbb C^N: \arg(\widehat{\x}(\omega)) \not= 0
\mbox{ for at most $k$ frequencies $\omega$}\}.
\end{equation}
We have to find the projection on $\A$. 
Given $\y$, we arrange $|{\rm Im}(\widehat{\y}(\omega_0))| \leq |{\rm Im}(\widehat{\y}(\omega_1))|\leq \dots \leq |{\rm Im}(\widehat{\y}(\omega_{N-1}))|$
for a permutation $\omega_0,\dots,\omega_{N-1}$ of $0,\dots,N-1$. Then $\x$ defined by
$\widehat{\x}(\omega_0)={\rm Re}(\widehat{\y}(\omega_0)), \dots, \widehat{\x}(\omega_{k-1})={\rm Re}(\widehat{\y}(\omega_{k-1}))$,  $\widehat{\x}(\omega_k)=\widehat{\y}(\omega_k),\dots,\widehat{\x}(\omega_{N-1})=\widehat{\y}(\omega_{N-1})$,
satisfies $\x \in P_\A(\y)$. This situation is now similar to sparsity in the physical domain.
Let $\widehat{S}\subset \{0,\dots,N-1\}$ denote subsets of cardinal $|\widehat{S}|=N-k$, and
let $P_{\widehat{S}}$ be the projection on the 
linear subspace $\{\y\in \mathbb C^N: \widehat{\y}(\omega) \in \mathbb R \mbox{ for all $\omega\in \widehat{S}$}\}$.
That is $P_{\widehat{S}}(\y) = ({\bf 1}-{\bf 1}_{\widehat{S}})\cdot \y + {\bf 1}_{\widehat{S}}\cdot {\rm Re}(\y)$.
Then $P_\A(\y) \subset \bigcup \{P_{\widehat{S}}(\y): |\widehat{S}|=N-k\}$, and for every $\y$ there exists a
set $\widehat{\frak{S}}$ of such $\widehat{S}$, depending on $\y$,  such that
$P_\A(\y) = \bigcup \{P_{\widehat{S}}(\y): \widehat{S}\in \widehat{\frak{S}}\}$. 

\begin{corollary}
Let $\x_n,\y_n$ be the Gerchberg-Saxton sequence for the sparse phase prior {\rm (\ref{sparse_phase})}.
Suppose $\A$ allows a guess better than ${\bf 0}$, from which
error reduction is started. Then the $\y_n$ converge toward a unique phase retrieval $\y^*$ with speed $O(n^{-\rho})$ for some
$\rho > 0$.
The $\x_n$ admit a finite set of accumulation points, each with sparse phase, and having $\y^*$
as their phase retrieval.
\hfill $\square$
\end{corollary}

It is again clear that when $|{\rm Im}(\y^*(t_{k-1}))| < |{\rm Im}(\y^*(t_k))|$, then the entire sequence
$\x_n$ converges, and the same is  true when $\x_n-\x_{n-1} \to 0$.

\begin{remark}
For the feasible case $A\cap B\not=\emptyset$ it has often been argued in the literature, see e.g. the essai \cite{luke_nonsense}, that convergence 
of alternating projections and Gerchberg-Saxton error reduction should 
be linear as a rule.  Typical supporting arguments are as follows:  $A,B$ drawn randomly,
will almost always intersect transversally. Or in the same vein: Even when
$A,B$ happen to intersect tangentially (as opposed to transversally), the slightest perturbation of their mutual position
would countermand this and lead back to transversality. 
 Even if one agrees with this reasoning, one should be aware that this does by no means resolve the dilemma of the phase retrieval
 literature \cite{luke_nonsense}. Namely, transversality is not a useful convergence criterion, because it is
 impossible to check it in practical situations. (Readers may convince themselves of the validity of our argument by trying to
 prove transversality of $\A \cap \B\not=\emptyset$ in any of the practical situations of this section.)
 For the feasible case, the only practically useful criterion for convergence of Gerchberg-Saxton error reduction
ever published is \cite{noll}. Our present contribution completes this picture by providing the very first verifiable conditions in the 
general case $r^* \geq 0$.
\end{remark}

\section{Cylinder and spiral}
\label{spiral}


In this section we show that Gerchberg-Saxton error reduction, even though convergent in natural situations, 
may fail to converge
even in the feasible case when the constraint set $\A$ is sufficiently pathological. We use an example constructed in \cite{douglas}, which we briefly
recall. We consider the cylinder mantle
\begin{equation}
    \label{B}
B=
\{x \in \mathbb R^3: x_1^2 + x_2^2 = 1, 0 \leq x_3 \leq 1\}
\end{equation}
the circle 
\begin{equation}
\label{F}
F=\{(\cos t, \sin t,0): t \geq 0\},
\end{equation}
and the logarithmic spiral
\begin{equation}
    \label{A}
A = \{((1+e^{-t})\cos t,(1+e^{-t})\sin t, e^{-t/2}): t \geq 0 \} \cup F
\end{equation}
winding around the cylinder with $A \cap B = F$. 
Alternating projections between the sets $A, B$ have been analyzed in \cite{douglas}, where in addition
a picture is available.  The findings
can be summarized as follows:
\begin{lemma}
{\rm (See \cite[Cor. 2]{douglas})}.
Every alternating sequence $a_k,b_k$ between cylinder mantle $B$ and spiral $A$, started at $a_1\in A\setminus F$, winds infinitely  often
around the cylinder, satisfies $a_k-a_{k+1} \to 0$, $b_k-b_{k+1}\to 0$, $a_k-b_k\to 0$, but fails to converge and 
its set of accumulation points is $F$. 
\hfill $\square$
\end{lemma}

\begin{remark}
The only hypothesis from Theorem \ref{theorem1} which fails  here is the angle condition, which is thereby shown to
be essential. 
Note that we may consider the sequence $a_k,b_k$ as alternating between the spiral and the solid cylinder co$(B)$, which is convex,
so the pathological behavior is caused by the spiral.
While $A$ is not prox-regular, we can see that the projector $P_A$ is single-valued and even Lipschitz at the points of $B$. 
This can be seen from an estimate obtained in \cite{douglas}. Suppose $P_A(b(t))=a(\tau)$,
where $b(t)=(\cos t,\sin t,e^{-t/2})\in B$ and $a(\tau)=((1+e^{-\tau})\cos \tau, (1+e^{-\tau})\sin \tau,e^{-\tau/2})\in A$
and $\tau(t) = {\rm argmin}_\tau \|b(t)- a(\tau)\|$, then
$t < \tau(t) < t-2\ln (1-e^{-t/2})$ from \cite{douglas}, which shows that $t \mapsto \tau(t)$ is Lipschitz. 
\end{remark}

\begin{remark}
The projector $P_A$ is certainly locally Lipschitz on a neighborhood of $B^s$ if $A$ is prox-regular and $B^s$
is within reach. The case of the spiral $A$, which is not prox-regular, shows that Lipschitz behavior of  $P_A|B^s$ is a
considerably weaker requirement, but sufficient to imply convergence of the $A$-sequence, provided the $B$-sequence converges. 
In particular, for the spiral $P_A$ is locally Lipschitz on $B^s$, but not on a neighborhood of $B^s$. 
This leads to the following open problem:
Find compact prox-regular sets $A,B$ with non-empty intersection and an alternating sequence $a_k,b_k$ with $a_k-b_k\to 0$,
$a_k-a_{k-1}\to 0$, 
which fails to converge. We know that at least one of the sets must fail to be subanalytic. 
\end{remark}

We use this example to construct an instance of  Gerchberg-Saxton error reduction, where convergence to a single limit fails.
Consider an unknown image $\x(t)$ with two pixels $t=0,1$, where amplitude
measurements of the discrete Fourier transform
\begin{equation}
    \label{fourier}
\widehat{\x}(\omega) = \frac{1}{\sqrt{2}} \sum_{t=0}^1 e^{i\pi t\omega} \x(t), \; \omega = 0,1
\end{equation}
are available under the form
\begin{equation}
\label{magnitude}
|\widehat{\x}(0)| = 1, \quad |\widehat{\x}(1)| = 1.
\end{equation}
This corresponds to the Fourier magnitude set
\begin{equation}
    \label{fat_B_new}
\B = \{\x \in \mathbb C^2: |\widehat{\x}(0)| = 1, |\widehat{\x}(1)| = 1\}.
\end{equation}
Since unique reconstruction of $\x(t)$ based on these measurements is not possible, the following prior information
is added. The unknown source is assumed to belong to
the prior set
\begin{align}
\label{fat_A}
\notag
\A = \big\{\x \in \mathbb C^2: |\widehat{\x}(0)| = 1 + &\left({\rm Re}\, \widehat{\x}(1) \right)^2,  |\widehat{\x}(1)| = 1,  0 \leq {\rm Re}\,\widehat{\x}(1)\leq 1,\\
&\left.{\rm Re}\, \widehat{\x}(0)=(1+{\rm Re}\, \widehat{\x}(1)^2)\cos \left( \ln {\rm Re}\, \widehat{\x}(1)\right)
\right\} \cup {\bf F},
\end{align}
where
\begin{equation}
    \label{fat_F1}
    {\bf F} = \{\x\in \mathbb C^2: |\widehat{\x}(0)|=1, {\rm Re}\,\widehat{\x}(1)=0, {\rm Im}\,\widehat{\x}(1)=1\}.
\end{equation}
Now any Gerchberg-Saxton sequence $\x_k\in \A,\y_k\in \B$ corresponds to
a unique alternating sequence $a_k\in A$, $b_k\in B$. Let $\mathscr F$
be the Fourier transform (\ref{fourier}), $\mathscr F'$ the inverse Fourier transform,
$\mathscr P$ the projector $\x\in \mathbb C^2\to ({\rm Re}\, x(0), {\rm Im}\, x(0),{\rm Re}\, x(1)) \in \mathbb R^3$, $\mathscr P'$ its adjoint
the inclusion
$x\in \mathbb R^3 \to (x_1+ix_2,x_3+i0)\in \mathbb C^2$. Then we have
\begin{equation}
    \label{link}
P_\A = \mathscr F \circ \mathscr P' \circ P_A \circ \mathscr P \circ \mathscr F', \quad
P_\B = \mathscr F \circ \mathscr P' \circ P_B \circ \mathscr P \circ \mathscr F'.
\end{equation}
All we have to see is that (\ref{fat_A}) is just a way of encoding the spiral (\ref{A}) in frequency coordinates,
and bearing in mind that the fourth coordinate is fixed throughout. 
In other words, $A = \mathscr P(\mathscr F'(\A))$, and $\A= \mathscr F ( \mathscr P'(A))$, and the same for $B,\B$. 
Our findings, based on \cite[Thm. 3]{douglas}, 
are now summarized by the following:

\begin{theorem}
Gerchberg-Saxton error reduction for the two pixel reconstruction problem {\rm (\ref{fourier}), (\ref{magnitude})} 
with prior information  {\rm (\ref{fat_A})}
fails to converge even though $\x_k-\x_{k+1}\to 0$, $\y_k-\y_{k+1}\to 0$, $\x_k-\y_k\to 0$. Every $\x^*\in {\bf F}$ 
is an accumulation point of the sequences $\x_k,\y_k$ and represents a possible exact solution of the phase retrieval problem.
\hfill $\square$
\end{theorem}

\section{Fienup's HIO-algorithm for phase retrieval}
\label{sect_hio}
Our construction can be used to show failure of convergence of other methods
used in phase retrieval, like hybrid input-output (HIO), relaxed averaged alternating reflections (RAAR), relaxed reflect reflect (RRR),
as those include the Douglas-Rachford algorithm for specific parameter values. We consider the Douglas-Rachford algorithm
\[
\x^+ = \x + P_{\A} (2P_{\B}(\x) - \x)  - P_{\B}(\x) = \textstyle\frac{1}{2} \left(R_{\A}R_{\B}   +I\right)(\x),
\]
where as before $\B$ is the magnitude set (\ref{magnitude}), and $\A$ gives prior information. We use again
\cite[Thm. 3]{douglas} to construct an example of failure of convergence.

Consider again the cylinder mantle $B$, but choose as set $A$ a double spiral
defined as follows:
\begin{equation}
\label{double_A}
a_\pm(t) = \left((1 \pm e^{-t}) \cos t,(1\pm e^{-t})\sin t, e^{-t/2}   \right) \in \mathbb R^3
\end{equation}
where $A_\pm = \{a_\pm(t): t \geq 0\}$ and $A = A_+ \cup A_- \cup F$.
The inner and outer spirals are mutual reflections of each other with respect to the cylinder mantle.
If we denote $b(t)\in B$ the projection of the two spirals on the mantle, then we obtain three curves
winding down inside, on, and around the cylinder toward the circle $F$ (see the picture in \cite{douglas}). 
If one starts a Douglas-Rachford iteration at some point
$x_1=a_-(t_1)\in A_-$ with $t_1 >0$  on the inner spiral, then $P_B(x_1) = b(t_1)$, hence $R_B(x_1) =a_+(t_1)\in A_+\subset A$,
and therefore $x_2 = (x_1+a_+(t_1))/2 = (a_-(t_1)+a_+(t_1))/2 = b(t_1) \in B$, which ends the first step of the DR-algorithm.
Now the second step starts at $x_2\in B$. The reflection in $B$ changes nothing $R_B(x_2)=x_2$, while reflection
in $A$ needs $P_A(x_2)=P_A(b(t_1))$, and as shown in \cite{douglas}, this projects always onto the inner spiral $A_-$, that is, we get
$P_A(b(t_1)) = P_{A_-}(b(t_1))= a_-(t_2)$ for some $t_2 > t_1$. 
Then $R_A(x_2) = 2a_-(t_2)-x_2$, which means $x_3 = a_-(t_2)\in A_-\subset A$. Hence after two DR-steps we are back
to the situation at the beginning, but at a slightly increased parameter value $t_2 > t_1$. 

As further shown in \cite{douglas}, the sequence $t_k$ so defined
satisfies $t_k\to \infty$ and $0\leq t_k-t_{k-1} \to 0$. That means, 
$$x_{2k-1}=a_-(t_k), x_{2k}= b(t_k)$$ 
and the $x_k$ fail to converge and wind around the cylinder
in the same way as the alternating projection sequence between $B$ and the inner spiral $A_-$. All points in $F$ are accumulation points
of the DR-sequence and also of the shadow sequences.

Now we lift this to produce a counterexample in the context of phase retrieval, using the same
method as in section \ref{spiral}.
We interpret the situation from the point of view of
the phase retrieval problem (\ref{fourier}), (\ref{magnitude}).
Since this is under-determined, we add the following prior information about 
$\x$, which is just a way to lift the double spiral $A$ into $\mathbb C^2$:
\begin{align}
\label{fat_Apm}
\notag
\A = \big\{\x \in \mathbb C^2: |\widehat{\x}(0)| = 1 \pm &\left({\rm Re}\, \widehat{\x}(1) \right)^2,  |\widehat{\x}(1)| = 1,  0 \leq {\rm Re}\,\widehat{\x}(1)\leq 1,\\
&\left.{\rm Re}\, \widehat{\x}(0)=(1\pm{\rm Re}\, \widehat{\x}(1)^2)\cos \left( \ln {\rm Re}\, \widehat{\x}(1)\right)
\right\} \cup {\bf F},
\end{align}
where ${\bf F}$ is as before.
Using (\ref{link}), we see that any Douglas-Rachford sequence for $\A,\B$ corresponds to a unique
Douglas-Rachford sequence for $A,B$. Therefore, based on \cite[Thm. 3]{douglas}, we derive the following

\begin{theorem}
The Fienup phase retrieval algorithm HIO for the two pixel reconstruction problem {\rm (\ref{fourier}), (\ref{magnitude})} 
with prior information  {\rm (\ref{fat_Apm})} just as well as the RAAR and RRR variants
fail to converge even though $\x_n-\x_{n+1}\to 0$, $\y_n-\y_{n+1}\to 0$, $\x_n-\y_n\to 0$. Every $\x^*\in {\bf F}$ 
is an accumulation point of the sequences $\x_n,\y_n$ and represents a possible exact solution of the phase retrieval problem.
\hfill $\square$
\end{theorem}

\begin{remark}
Recall that the DR-algorithm is asymmetric with regard to $A,B$,
so one may wonder whether changing order and using
$\frac{1}{2} (R_{B}R_{A}+I)$ gives still failure of convergence. 
We now reflect first in the double spiral, then in the cylinder mantle, and then average.
Starting at $x_1=a_+(t_1)\in A_+$ in the outer spiral, we get $R_A(x_1)=x_1$, and then $R_B(R_A(x_1)) = a_-(t_1)\in A_-$, so that
averaging gives $x_2 = (a_-(t_1)+a_+(t_1))/2 = b(t_1)\in B$. Now $R_A(x_2)=2a_-(t_2)-x_2$ for $a_-(t_2) = P_{A_-}(b(t_1))$,
and then $R_B(R_A(x_2)) = 2a_+(t_2)-x_2$, so that averaging gives $x_3 = a_+(t_2)$, when we are back ion $A_+$ with a slightly
enlarges $t_2 > t_1$. So here we can see that the DR-iterates follow
alternating projections between $B$ and $A_+$, and convergence fails again.
\end{remark}

\begin{remark}
Convergence theory of the DR-algorithm for phase retrieval is even less advanced than for alternating projections.
Even the most pertinent currently available result \cite{hung} needs some form of transversality of $\A \cap \B \not=\emptyset$, which
as we argued above is impossible to verify in practice. It is therefore of interest to dispose at least of a limiting counterexample. 
\end{remark}

\section{Gaussian EM-algorithm revisited}
\label{EM}
The following situation involves a special case of the EM-algorithm for gaussian
random vectors with unknown mean and known variance. It can be used in
image restoration methods; cf. Bauschke {\em et al.} \cite{dynEM}, where this has been applied to emission tomography.

We consider a random vector $Y$ with joint  distribution $f_Y(y|x)$ representing the incomplete data space,
where the law depends linearly on a parameter $x\in \Omega \subset \mathbb R^n$ via
\[
E(Y_j|x) = \sum_{i=1}^n c_{ji} x_i, \quad j=1,\dots,m.
\]
Defining $C=(c_{ji})$, this can be written as $E(Y|x) = Cx$, where $C$ may typically
lead to a certain loss of information.
Suppose a sample $y\in \mathbb R^m$ of $Y$ is given,
then the maximum likelihood estimation problem is
\begin{eqnarray*}
\begin{array}{ll}
\mbox{minimize} & -\ln f_Y(y|x) \\
\mbox{subject to} & x \in  \Omega
\end{array}
\end{eqnarray*}
Now assume that $Z$ is a random vector of size $nm$ and joint distribution $f_Z(z|x)$, depending on the parameter $x\in \Omega$, 
representing the complete data space, where
\[
E(Z_{ji}|x) = c_{ji} x_i.
\]
Introducing the linear operator $\Gamma: x \mapsto c_{ji}x_i$, this reads $E(Z|x) = \Gamma x$. 
Assuming that maximum likelihood estimation is easier in complete data space, one applies 
the well-known EM-algorithm, which is the following alternating procedure:

\begin{algorithm}[!ht]
\caption{\!\!{\bf .} EM-algorithm}
\begin{algorithmic}[1]
\STEP{E-step} Given current parameter estimate $x^{(t)}\in \Omega$, supply completed data by computing
conditional expectation
$$z^{(t)} = E\left(Z|Z_{j1}+\dots+Z_{jn} =y_j,x^{(t)}\right).$$

\STEP{M-step} Given completed data sample $z^{(t)}$ for $Z$, perform maximum likelihood
estimation in complete data space
\begin{eqnarray*}
\begin{array}{ll}
\mbox{minimize} & -\ln f_Z(z^{(t)}|x) \\
\mbox{subject to} & x \in  \Omega
\end{array}
\end{eqnarray*}
The result is the new parameter estimate $x^{(t+1)}\in \Omega$.

\end{algorithmic}
\end{algorithm}

If we consider the case where $Y,Z$ are independent and normally distributed with known variance $\sigma^2$, 
the E-step has the explicit form
\begin{equation}
    \label{Estep}
z_{ji}^{(t)}= \frac{1}{n} y_j+c_{ji} x_i^{(t)} - \frac{1}{n} \sum_{i'=1}^n c_{ji'} x_{i'}^{(t)},
\end{equation}
which is the orthogonal projection of the estimate $v^{(t)}\in \mathbb R^{nm}$ with $v_{ji}^{(t)}=c_{ji} x_i^{(t)}$ onto the set
$B = \{z\in \mathbb R^{nm}: z_{j1}+\dots+z_{jn} = y_j, j=1,\dots,m\}$. At the same time, the M-step
as well turns out to be an orthogonal projection, namely, the orthogonal projection of $z^{(t)}$  onto the set
\[
A = \{v \in \mathbb R^{mn}: v = \Gamma x \mbox{ for some $x\in \Omega$}\},
\]
where $v^{(t+1)} \in P_A(z^{(t)})$. 
This leads now to the following
\begin{theorem}
Suppose $\Omega$ is a bounded closed subanalytic set, and consider sequences $z^{(t)}, x^{(t)}$ and $v^{(t)}$ generated by the
Gaussian EM-algorithm with known variance $\sigma^2$. Then the sequence $z^{(t)}$ converges to  a
limit $z^*$ with rate $\|z^{(t)} - z^*\| = O(t^{-\rho})$ for some $\rho > 0$. Moreover, if
$x^*\in \Omega$ is any of the accumulation points of the $x^{(t)}$, then 
$z^*=E(Z|Z_{j1}+\dots+Z_{jn}=y_j,x^*)=\frac{1}{n} y_j+c_{ji} x_i^* - \frac{1}{n} \sum_{i'=1}^n c_{ji'} x_{i'}^*$, $x^*$ is a critical point of the
complete data space maximum likelihood estimation problem
$\min\{-\ln f_Z(z^*|x): x\in \Omega\}$, and 
$z_{ji}^* - c_{ji}x_i^*$ is independent of $i$ for every $j$.
\end{theorem}

\begin{proof}
The main convergence theorem gives convergence $z^{(t)} \to z^*$ with rate $O(t^{-\rho})$ if we consider that
$B$, being an affine subspace, has infinite reach, while $A$ is subanalytic. The latter
follows because  $A$ can be defined equivalently by
the relations $(v_{j1}/c_{j1})\in \Omega$, $v_{ji}c_{1i} = v_{1i}c_{ki}$. 

Clearly $v^{(t)} = \Gamma x^{(t)}$ implies $v^*=\Gamma x^*$ for every accumulation point $x^*\in \Omega$ of the $x^{(t)}\in \Omega$. 
Now from (\ref{Estep}) $z_{ji}^* - \frac{1}{n} y_j = c_{ji}x_i^* - \frac{1}{n} \sum_{i'=1}^n c_{ji'} x_{i'}^*$ for all $i,j$
we see that for two accumulation points $x^*_1,x^*_2\in \Omega$ the shift
$x^*_1-x^*_2$ is in the kernel of the operator $(\Gamma x)_{ji}-\frac{1}{n} (Cx)_j$, because the left hand term is the same for every $x^*$. 
It also follows that $z_{ji}^* - c_{ji}x_i^*= \frac{1}{n}y_j - \frac{1}{n} \sum_{i'=1}^n c_{ji'} x_{i'}^*$ is independent of $i$ for every $j$.
\hfill $\square$
\end{proof}

\begin{remark}
The result is interesting for two reasons. Firstly, even for this very elementary case no convergence result
has been known for a non-convex parameter set $\Omega$ since the 1970s. For a convex $\Omega$ convergence follows of course
from the classical convergence result \cite{bauschke-survey}. The second aspect is that some insight into
the speed of convergence is provided. This has been a point of vivid interest in various forms of the EM-algorithm,
and our result suggests that the speed $O(t^{-\rho})$ can be extremely slow. Note also that the M-step may be optimized locally,
which is convenient when $\Omega$ is 'curved'. 
\end{remark}

\begin{remark}
In \cite{dynEM} this method is applied to dynamic SPECT imaging with slow camera rotation, where
$x_{ik} = x_i(t_k)$ represents the unknown tracer activity in voxel $i$ at angular camera position $\theta_k=k\Delta\theta$ 
at time $t_k=k\Delta t$, while $y_{jk} = y_j(t_k)$ is  the sinogram, 
i.e., the activity received in camera bin $j$ at position $\theta_k$ and time $t_k$,
with $C$ the linear operator representing
camera geometry and collimator specifications. The artificial complete data $z_{ijk}=z_{ij}(t_k)$ 
represent that part of the activity emanating from voxel $i$ toward camera bin $j$ at time $t_k$ and camera position $\theta_k$.
Due to missing data, a dynamic model of the form $x_i(t) = A_ie^{-\lambda_i t} + B_ie^{-\mu_it} + C_i$ is imposed,
giving rise to the non-convex set $\Omega$. 

In \cite{maeght} a Prony type model
$x_{ik} -\alpha_{1i} x_{i,k-2} -\alpha_{2i} x_{i,k-1} - \alpha_{1i}=0$ is used instead to implement
a constraint on the tracer dynamics, giving rise to yet another non-convex parameter set $\Omega$, to which our convergence result applies.
\end{remark}

\begin{remark}
The averaged projection method can be obtained as a special case of the Gaussian EM-algorithm. Let $\Omega= C_1 \times (-C_2) \times C_3 \times\dots \times (\pm C_m)$
and $\Gamma = I$. Then the M-step is equivalent to the coordinatewise projection
$x_i \in P_{C_i}(x)$. For the E-step we we have to come up with the operator $C$, which
we model as $x_1-x_2=0$, $-x_2+x_3=0, \dots$. Then averaging is the E-step (\ref{Estep}) with data vector $y=0$. 
\end{remark}

\section{Structured low-rank approximation}
\label{sect-cadzow}

Structured
low-rank approximation 
has  the general form:
\begin{equation}
\label{cadzow}
\mbox{find a matrix $S \in A$ such that
{\rm rank}$(S) \leq r$},
\end{equation}
where $A \subset \mathbb C^{n \times m}$ is a closed set of {\it structured} $n \times m$-matrices, and $r \ll \min(n,m)$.  Letting
$B = \{R \in \mathbb C^{n \times m}: {\rm rank}(R) \leq r\}$,
we seek a matrix $S \in A \cap B$ which has structure {\it and} low rank,
and  this is addressed via alternating projections between $A,B$ in the euclidean space $\mathbb C^{n \times m}$,
equipped with the Frobenius norm $\|\cdot\|_F$. Motivated by \cite{cadzow1,cadzow2}, see Example \ref{basic} below, 
we call the corresponding alternating sequence
$$
R_k \in P_B(S_k), S_{k+1} \in P_A(R_k), k=1,2,\dots
$$
a {\it Cadzow alternating sequence}, and its limit $S^* \in A \cap B$ a Cadzow solution of (\ref{cadzow}).

Projections
 $R \in P_B(S)$  on the low-rank set are obtained by
singular value decomposition (SVD). Let $S = U \Sigma V^T$ with $\Sigma = {\rm diag}(\sigma_1, \dots, \sigma_{\min(n,m)})$ and $\sigma_1 \geq \sigma_2  \geq \dots$ be an SVD
of $S$,
then every $r$-truncation  $\Sigma' = {\rm diag}(\sigma_1,\dots,\sigma_r,0,\dots)$  of $\Sigma$, i.e.,  keeping  $r$ largest 
singular values and zeroing the others, gives rise to an element 
$R=U\Sigma'V^T\in  P_B(S)$. 
Assuming $\sigma_1 \geq \dots \geq \sigma_{k-1} > \sigma_k = \dots = \sigma_r = \dots = \sigma_\ell > \sigma_{\ell+1} \geq \dots$
for certain $k \leq r \leq \ell$, 
we have ${\ell-k+1\choose r-k+1}$ possibilities to choose such an $r$-truncation $\Sigma'$ of $\Sigma$, and since each gives rise to
a unique $R = U\Sigma'V^T$,  this is the cardinality of
$P_B(S)$. Since $\|\Sigma'-\Sigma''\|_F \geq 2\sigma_r$ for any two $r$-truncations, it follows that $B$ has positive reach $\sigma_r$
at every projected point
$R$, and on $\mathcal B(R,\sigma_r)$ the projection $P_B$ is single valued.
This leads now to our first
result.

\begin{theorem}
\label{cadzow1}
Let $A \subset \mathbb C^{n\times m}$ be a closed subanalytic set of structured matrices, $B$ matrices of rank $\leq r$.
Let $R_k,S_{k}$ be a bounded Cadzow alternating sequence with gap $(A^*,B^*,r^*)$, and suppose
$r^* < \sigma^*_r$ for the reach $\sigma^*_r = \min\{\sigma_r(R): R \in B^*\}$ of $B^*$. Then the $R_k$ converge to a low rank matrix
$R^*\in B$ with speed  $\|R_k-R^*\|_F =O(k^{-\rho})$ for some $\rho >0$. All accumulation points $S^*\in A$ of the sequence
$S_k$ are structured matrices, and all admit $R^*$ as their low-rank approximation. If in addition $r^*=0$, then $S_k\to R^*\in A \cap B$
with the same speed.
\end{theorem} 

\begin{proof}
Since $R\in B$ iff the determinants of all $(r+1)\times(r+1)$-minors of $R$ vanish,  $B$ is 
the solution set of a finite number of polynomial equations, i.e., a semi-algebraic variety, also known as determinantal variety of dimension
$r(n+m-r)$;  \cite{determinantal}. 
Since $A$ is subanalytic by hypothesis and $B$ is prox-regular and closed, we may apply our convergence theory.
For $r^*>0$ 
we use Corollary \ref{global}, whereas the case $r^*=0$ is already contained in \cite[Thm.1]{noll}.
\hfill $\square$
\end{proof}

The limitation here is that the attracting neighborhoods $\mathcal N(A^*,\delta), \mathcal N(B^*,\delta)$ of the gap $(A^*,B^*,r^*)$ may be small, 
as $\delta$ 
depends on the reach $\sigma^*_r$ of $B$ at the $R^*\in B^*$. 
Often we can do better, since usually the structure 
set $A$ has additional properties.

\begin{theorem}
\label{structure}
Let the structure set $A$ be closed subanalytic and prox-regular. Let $R_k,S_k$ be a bounded Cadzow sequence
with gap $(B^*,A^*,r^*)$, where $r^* <  \rho^*$ for the reach $\rho^*$ of $A^*$. Then the $S_k$ converge to a structured matrix $S^*\in A$ with speed
$\|S_k-S^*\|_F = O(k^{-\rho})$
for some $\rho >0$. The sequence $R_k$ has a finite set of accumulation points $R^*\in B$, and each $R^*$ is a low-rank
 approximation of $S^*$. If in addition $r^* < \sigma_r^*$ for the $r$th singular value $\sigma_r^*$ of $S^*$, then the sequence $R_k$ converges to
 a unique low-rank approximation $R^*$ of $S^*$. The same is true when $\sigma_r^* > \sigma_{r+1}^*$, or when $\limsup_{k\to \infty} \|R_k-R_{k-1}\|_F <2\sigma_r^*$. 
\end{theorem}

\begin{proof}
Here we apply the main convergence theorem to the dual gap $(B^*,A^*,r^*)$, where it is now the reach of $A^*$
that matters. We obtain convergence $S_k\to S^*\in A$ from the main convergence theorem. The specific structure of $P_B$  assures that the set of accumulation points
$R^*$ of the $R_k$ is finite,
and clearly every such $R^*$ is a low rank approximation of the same $S^*$.

For $\sigma_r^* > \sigma_{r+1}^*$
the projection $R^*=P_B(S^*)$ is single valued, hence the $R_k$ converge to $R^*$,
and the same is true
for $\|R_{k}-R_{k-1}\|_F \leq 2\sigma_r^* -\epsilon$ for $k \geq k_0$, because the distance between two elements $R^*\in B^*$ is $2\sigma_r^*$, hence the sequence
$R_k$  can then have 
only one accumulation point, to which it converges.
Finally, for $r^* < \sigma_r^*$, the projection $P_B$ is single-valued and locally Lipschitz, so the $R_k$ converge to
$R^*=P_B(S^*)$ with the same speed $\|R_k-R^*\|_F = O(k^{-\rho})$.
\hfill $\square$
\end{proof}

In most applications the set $A$ is convex and subanalytic, or even affine, in which case the sequence $S_k$, when bounded,  converges from an arbitrary starting point,
while the $R_k$ still admit their finite set of accumulation points as described above. In the literature Cadzow's method is usually
presented for affine $A$, but we use the term in a broader sense, because we get convergence for a much broader class of structures $A$. 

\begin{example}
\label{basic}
Historically the first application is {\it Cadzow's basic algorithm}  in signal de-noising; cf. \cite{cadzow1,cadzow2}.  Given a Toeplitz
matrix $\widetilde{T}\in \mathbb C^{n \times m}$, encoding
a noisy signal, 
one wishes to solve the problem:
\begin{eqnarray}
\label{denoise}
\begin{array}{ll}
\mbox{minimize} & \|T-\widetilde{T}\|_F \\
\mbox{subject to} & {\rm rank}(T) \leq r, \mbox{ $T$ Toeplitz }
\end{array}
\end{eqnarray}
where the de-noised signal is encoded in the solution $T$ of (\ref{denoise}). Letting $A$ be
the set of Toeplitz matrices, Cadzow's heuristic \cite{cadzow1} consists in projecting alternatively on $A,B$, starting at $\widetilde{T}\in A$,
$R_1 \in P_B(\widetilde{T})$,
$T_{k+1}=P_A(R_k),R_{k}\in P_B(T_k)$. 
Here
$T=P_A(R)$,  the nearest Toeplitz matrix to a given matrix $R$, is  obtained explicitly by fixing the
value in each diagonal of $T$ as the average of the values in the corresponding diagonal in $R$:
$T_{1+k,i+k} = T_{1i}= (R_{1i}+R_{2,i+1}+\dots+R_{n-i+1,n})/(n-i+1)$.
This example motivated our nomenclature.
\end{example}

\begin{corollary}
 {\bf (Global convergence for Cadzow)}.
 Let $A$ be closed convex and subanalytic. Then
 every bounded Cadzow sequence $S_k\in A$ converges
 to $S^*\in A$ with speed $\|S_k-S^*\|_F = O(k^{-\rho})$ for some $\rho >0$. The corresponding low rank $R_k \in B$ have a finite set of low-rank accumulation points
 $R_1^*\dots,R_N^*\in B$, where $S^*$ is the nearest structured matrix to each $R_i^*$, and where each
 $R_i^*$ is a $r$-truncated SVD of $S^*$. Convergence of the $R_k$ to a single $R^*$ occurs under any of the additional conditions in Theorem {\rm \ref{structure}}.
\end{corollary}

\begin{proof}
 The set $A$ is subanalytic and convex, hence of infinite reach, and since $B$ is subanalytic, the sequence $S_k$ is now convergent with limit $S^*\in A$
 for an arbitrary starting point. 
 All accumulation points $R^*_i$ of the sequence $R_k$ satisfy $R_i^*\in P_B(S^*)$, and we have $S^*=P_A(R_i^*)$ for every $i$. From the discussion above
 we know that there are only finitely many such accumulation points.
 
 Convergence to a single low rank $R^*$ occurs under any of the
 conditions in Theorems \ref{cadzow1},\ref{structure}, that is, when $\sigma_r^* > \sigma_{r+1}^*$, or when the $R_k$ come within reach of $B$,
 or again when $\limsup_{k \to \infty} \|R_k-R_{k-1}\|_F <2 \sigma_r^*$.
 \hfill $\square$ 
 \end{proof}

\begin{corollary}
{\bf (Feasible case for Cadzow)}.
Let $S^\sharp \in A \cap B$ be a Cadzow solution to the 
low rank structured approximation problem
{\rm (\ref{cadzow})}. There exists $\delta > 0$ such that every Cadzow sequence $S_k,R_k$ which enters the $\delta$-neighborhood of
$S^\sharp$
 converges to a Cadzow solution $S^* \in A \cap B$ with speed $O(k^{-\rho})$ for some $\rho > 0$. 
 \hfill $\square$
 \end{corollary}

 \begin{remark}
 In the case of Cadzow's basic sequence in Example \ref{basic}  it is important to be allowed the starting point $\widetilde{T}$, because we want a restoration $T^*\in A \cap B$ 
 close
 to $\widetilde{T}$. The method is a heuristic, because even in the case of convergence $T_k,S_k\to T^*$ we do not get the exact solution of (\ref{denoise}).
 Convergence to the projection of the initial guess  on $A \cap B$ is only obtained for alternating projections between affine subspaces \cite{bauschke-survey}. 
 \end{remark}
 
\begin{remark}
Even for affine $A$ our convergence result  is new,  while in the feasible case $r^*=0$ convergence is already affirmed by \cite{initial,noll},
even though there this was not stated explicitly for the Cadzow case. 
\end{remark}

\begin{remark}
Convergence claims for Cadzow's basic method, and for more general affine structures  $A$,  have been made repeatedly
in the literature. None of the published arguments the author is aware of are tenable. Most authors claim that convergence is linear and follows form
\cite{malick}. We show by way of an example that
this is incorrect, because \cite{malick} requires the manifolds to intersect transversally, and this fails in general.
\end{remark}

\begin{example}
\label{ex2}
Consider the set $B$ of $2 \times 2$ matrices of rank $\leq 1$,   $B_1$ those of rank equal 1, 
$B = B_1 \cup \{0_{2\times 2}\}$. Let $\bar{Y}= \begin{bmatrix} 1&-1\\2&-2\end{bmatrix}\in B_1$
and parametrize
$Y\in B_1$ in the neighborhood of $\bar{Y}$  by a vector $y=(y_1,y_2,y_3)\in \mathbb R^3$ as 
$Y = \begin{bmatrix} y_1 &y_1y_2 \\ y_3& y_3y_2\end{bmatrix}$, where
$\bar{y}=(1,-1,2)$ gives $\bar{Y}$.  Let $y(t)=(y_1(t),y_2(t),y_3(t))\in \mathbb R^3$ be a smooth curve with $y(0)=\bar{y}$, then
$Y(t) = \begin{bmatrix} y_1(t) & y_1(t)y_2(t) \\ y_3(t) & y_3(t)y_2(t) \end{bmatrix}\in B_1$ is a smooth surface curve on $B_1$
near $Y(0)=\bar{Y}$.  
Its tangent vector is $$\dot{Y}(t) = \begin{bmatrix} \dot{y}_1(t) & \dot{y}_1(t)y_2(t) + y_1(t) \dot{y}_2(t) \\ \dot{y}_3(t) & \dot{y}_3(t)y_2(t)+y_3(t)\dot{y}_2(t)\end{bmatrix},$$
hence $\dot{Y}(0)= \begin{bmatrix} \dot{y}_1(0) & -\dot{y}_1(0)+\dot{y}_2(0)\\ \dot{y}_3(0)& -\dot{y}_3(0)+2 \dot{y}_2(0)\end{bmatrix}$ is 
an element of the tangent space to $B_1$ at $\bar{Y}$. If we choose $\dot{y}_1(0)=\dot{y}_2(0)=\dot{y}_3(0)=1$, then $\dot{Y} = \begin{bmatrix} 1&0\\1&1\end{bmatrix}$
is a tangent direction to $B_1$ at $\bar{Y}$. We have
$\bar{Y} + t\dot{Y} \not \in B$ for $t \not= 0$, which corroborates that $B_1$ is curved. Now let
$A$ be the affine set 
$A = \bar{Y} + \mathbb R \dot{Y}$ which defines our structure, then $A\cap B = \{\bar{Y}\}$, and the intersection is tangential, because
$\dot{Y}$ belongs to both tangent spaces. Indeed, $T_{\bar{Y}}A = \mathbb R\dot{Y}$, hence
$T_{\bar{Y}}A + T_{\bar{Y}}B_1 = T_{\bar{Y}}B_1\not= \mathbb R^{2\times 2}$,  the latter
since the tangent space of $B_1$ at $\bar{Y}$ is not the full $\mathbb R^{2\times 2}$, given that dim$(B_1)=3$. Hence $A,B$ do not intersect transversally at 
$\bar{Y}$. 

Convergence for this simple example can be derived from \cite{noll}, but results based on transversality do not apply. In particular, convergence
to $\bar{Y}$ can be proved to be sublinear, as can be confirmed numerically.
\end{example}

\begin{example}
We may expand on Example \ref{ex2} by choosing a second smooth curve $z(t)$ with $z(0)=\bar{y}$, now with $\dot{z}_1(0)=1$, $\dot{z}_2(0)=-1$,
$\dot{z}_3(0)=0$, which gives a second tangent $\dot{Z}=\begin{bmatrix} 1&-2\\0&-2\end{bmatrix}$ to $B_1$ at $\bar{Y}$. Then
$A = \bar{Y} + \mathbb R\dot{Y}+\mathbb R\dot{Z}$ is two-dimensional, but still $T_{\bar{Y}}A \subset T_{\bar{Y}}B_1$, and the intersection is again 
tangential with $A \cap B = \{\bar{Y}\}$.
\end{example}

 \begin{example}
 It should also be stressed that one has to assume that the Cadzow alternating sequence
 is bounded, because the low rank set $B$ may have asymptotes, so that Cadzow iterates may escape to infinity.
 We give an example again for $B \subset \mathbb R^{2 \times 2}$.  Choose the affine structure
 $A = \{ S \in \mathbb R^{2 \times 2}: S_{12}=S_{21} = 1, S_{22}=0\}$, then the Cadzow alternating sequence,
 started at $S^0$ with $S^0_{11}=1$
 produces $S^k\in A$ where $S^k_{11} \to \infty$, as can also be verified numerically. 
 \end{example}

\end{document}